\documentclass[11pt,letterpaper]{amsart}
\usepackage{times,comment}
\usepackage{amsmath,amssymb,amscd,amsthm, soul, color,mathabx}
\usepackage{latexsym}
\usepackage{epsfig}

\newtheorem{theorem}{Theorem}
\newtheorem{lemma}{Lemma}

\newtheorem{remark}{Remark}

\newcommand{\nab}{\langle\nabla\rangle}

\newcommand{\tn}[1]{\vertiii{#1}}

\newcommand{\n}[2]{{\left\| #1 \right\|}_{#2}}
\newcommand{\f}[2]{\frac{#1}{#2}}
\newcommand{\lan}[1]{\left\langle #1\right\rangle}
\newcommand{\pr}[1]{\left(#1\right)}
\newcommand{\wh}[1]{\widehat{#1}}
\newcommand{\wt}[1]{\widetilde{#1}}

\newcommand{\ve}{\varepsilon}

\newcommand{\si}{\sigma}

\newcommand{\vp}{\varphi}

\newcommand{\br}{{\mathbb R}}
\newcommand{\bt}{{\mathbb T}}

\newcommand{\rr}{\mathcal R}
\newcommand{\cs}{\mathcal S}
\newcommand{\cm}{\mathcal M}
\newcommand{\cf}{\mathcal F}

\newcommand{\cn}{{\mathcal N}}

\newcommand{\IM}{\textnormal{Im}\,}
\newcommand{\p}{\partial}

\newcommand{\ds}{\displaystyle}
\newcommand{\vertiii}[1]{{\left\vert\kern-0.25ex\left\vert\kern-0.25ex\left\vert #1 
    \right\vert\kern-0.25ex\right\vert\kern-0.25ex\right\vert}}
\begin{document}

\title
[Local Well-posedness of Periodic KdV]
{On Local Well-posedness of the Periodic Korteweg-de Vries Equation Below $H^{-\frac{1}{2}}(\mathbb{T})$}

\author{Ryan McConnell, Seungly Oh}
\thanks{The first author was partially supported by the NSF grant  DMS-2154031.}
\date{\today}



\begin{abstract} 
We utilize a modulation restricted normal form approach to establish local well-posedness of the periodic Korteweg-de Vries equation in $H^s(\mathbb{T})$ for $s> -\frac23$. This work creates an analogue of the mKdV result by Nakanishi, Takaoka, and Tsutsumi, \cite{NTT} for KdV, extending the currently best-known result of $s \geq -\frac12$ without utilizing the theory of complete integrability.
\end{abstract}

\maketitle
\section{Introduction}
We consider the real-valued periodic Korteweg-de Vries (KdV) equation:
\begin{equation}\label{eq:kdv original}
\left|\begin{array}{l} u_t + u_{xxx}= (u^2)_x\\\left.u\right|_{t=0}  = u_0 \in H^{s}(\bt) \end{array}\right.
\end{equation}
where we assume the mean-zero condition $\int_{\mathbb{T}} u_0\, dx =0$ and  $s\leq -1/2$. Equation~\eqref{eq:kdv original} satisfies conservation of mean,
\[
\int_\mathbb{T}u(x,t)\,dx = \int_\mathbb{T} u_0(x)\,dx,
\]
and hence we may assume that $u$ also satisfies the mean-zero condition.\\

To study local well-posedness for the periodic KdV equation, Bourgain \cite{Bour1, Bour2} created an $L^2_{t,x}$-based Sobolev-type space where the Fourier weight was adapted to the Airy free group.  This space is referred to as $X^{s,b}$ space, whose norm is defined by:
\[
\|u\|_{X^{s,b}} := \|e^{it\partial_x^3}u\|_{H^b_tH^s_x} = \|\langle \tau+n^3\rangle^b \langle n\rangle^s\widehat{u}(\tau,n)\|_{L^2_\tau\ell^2_n}.
\]
In order to prove the local-wellposedness using a standard contraction argument, one relies on the following bilinear estimate:
\begin{equation}\label{Eq: introKDVbad}
\left\|\frac{\langle n\rangle^{s}}{\langle \tau+n^3\rangle^{\frac12}}\cf_{t,x} \left[\p_x u^2\right](\tau,n)\right\|_{L^2_\tau\ell^2_n}\lesssim \|u\|_{X^{s,\frac12}}^2
\end{equation}
where $\cf_{t,x}$ represents the Fourier transform in $t$ and $x$-variables.  Utilizing this, Bourgain \cite{Bour2} proved that \eqref{eq:kdv original} is locally and globally well-posed in $H^s$ for $s\geq 0$. Kenig, Ponce, and Vega  \cite{KPV} extended this well-posedness result to  $s > -\f{1}{2}$, and  Colliander, Keel, Staffiliani, Takaoka, and Tao \cite{I1} established that the local well-posedness holds at the endpoint $s = -\f{1}{2}$ and also that the equation is globally well-posed for $s \geq -\f{1}{2}$.  This is currently the best well-posedness result for the periodic KdV without using the theory of complete integrability.  Failure of the bilinear estimate \eqref{Eq: introKDVbad} for $s<-\f{1}{2}$ necessitates a drastically different approach for studying local well-posedness in this regime.\\

On the other hand, the completely integrable structure for \eqref{eq:kdv original} is well-known, as it is the original model used in the development of the inverse-scattering method for solving nonlinear partial differential equations.  Due to this structure and symmetry inherent in KdV, a better well-posedness result is available.  In \cite{KT}, Kappeler and Topalov utilized the theory of inverse scattering method to prove that \eqref{eq:kdv original} is globally well-posed for $s\geq -1$.  In \cite{mol}, Molinet proved that \eqref{eq:kdv original} is locally illposed for $s<-1$, finally settling the question of sharp range of well-posedness for this equation.\\

The periodic modified KdV (mKdV) equation is closely related to \eqref{eq:kdv original} as the Miura transform ($u\mapsto \partial_xu^2+u^2$, \cite{MR252825}) maps solutions of the defocusing mKdV to those of the KdV.  Due to this transform, one expects results for the mKdV to again hold for the KdV, with the regularity lowered by exactly $1$.  In fact, Colliander, Keel, Staffiliani, Takaoka, and Tao showed the global well-posedness of mKdV for $s \geq \f{1}{2}$, which confirms this expectation.  Further, Kappeler and Topalov's inverse scattering approach \cite{KT} shows global well-posedness of mKdV for $s\geq 0$, and Molinet \cite{mol} proved the illposedness of mKdV for $s<0$.\\

While, heuristically, well-posedness results for the mKdV at level $s$ \textit{should} imply similar results for the KdV at level $s-1$, this connection cannot be immediately claimed since the Miura map is not invertible.  For studying the mKdV equation, we work with regularity $s\geq 0$, where necessary harmonic analysis estimates such as \eqref{Eq: introKDVbad} is easier to obtain.  On the other hand, such estimates become more technical (or impossible) when $s<0$.\\

It is remarkable that better results are known for local well-posedness of the periodic mKdV, for which there is presently no KdV counterpart.  In \cite{TT}, Takoaka, and Tsutsumi proved the local well-posedness of mKdV for $s> \f{3}{8}$ without relying on complete integratibility; and Nakanishi, Takoaka, and Tsutsumi \cite{NTT} extended this result to $s> \f{1}{3}$.  In the latter, the authors also show the existence of solutions for $s>\f{1}{4}$ without uniqueness.  In a separate theorem, uniqueness was recovered for $s> \f{1}{4}$ by imposing an additional condition on the initial data.\\

To appreciate these results fully, we should mention the work by Christ, Colliander and Tao \cite{cct} which shows that the solution map of periodic mKdV (respectively, that of the periodic KdV) is not uniformly continuous when $s<\f{1}{2}$ (respectively, $s< -\f{1}{2}$).  Since a standard contraction map argument for well-posedness implies uniform continuity for the solution map, this means that contraction map cannot be used in this setting.  To obtain well-posedness result of mKdV below $s=\f{1}{2}$, the authors constructed an $X^{s,b}$-type space with a Fourier weight dependent on the initial data.  This construction aided in the removal of a resonant nonlinear interaction, and we will follow a similar approach to establish the local well-posedness of periodic KdV below $s=-\f{1}{2}$.\\

As mentioned previously, the primary goal of this manuscript is to establish an analogue of this mKdV result for the periodic KdV equation.  Theorem~\ref{Theorem: 1} below yields a KdV analogue of mKdV local well-posedness \cite{NTT}, but with a caveat of small-data restriction.  We do not pursue the analogue of non-unique (or conditionally unique) existence for $s> \f{1}{4}$.  While we adapt the use of Bourgain space developed in \cite{TT, NTT}, we encounter unique challenges and complications, especially when obtaining multilinear estimates, which are quite different than those in the mKdV setting.  We will develop a novel approach to tackle these issues.  Our main result is summarized in the following:
 
\begin{theorem}\label{Theorem: 1}
    Let $u_0\in H^s(\mathbb{T})$ be real-valued and $-\f{2}{3} < s \leq -\f{1}{2}$. Given that $\n{u_0}{H^s}$ is sufficiently small, there is a space $Y = Y(u_0)$ and a time $T = T(\|u_0\|_{H^s}) > 0$ such that there is a unique solution $u\in C^0_t([0,T], H^{s}_x(\mathbb{T}))\cap Y$ to \eqref{eq:kdv original}. Moreover, the data-to-solution map is continuous.
\end{theorem}

We will briefly outline our approach to prove this theorem.  The main difficulty for working in the $s< -\f{1}{2}$ regime is the failure of  \eqref{Eq: introKDVbad}.  To work around this issue, we utilize the normal form transformations \cite{Oh, shatah} or, equivalently, differentiation-by-parts  \cite{titi, BurakNikos}, which have been used for unconditional well-posedness \cite{katoCanc, Mcc2KdV}, for global well-posedness \cite{I1}, and for non-linear smoothing \cite{ErdTz, Mcc2KdV, OhStef}. Heuristically, a normal form transformation enables one to rewrite \eqref{eq:kdv original} as
\begin{multline*}
(\partial_t-in^3)\widehat{u}(n,t) = \partial_t\left(e^{-itn^3}\sum_{n_1+n_2=n}\frac{1}{n_1n_2}\widehat{u}(n_1,t)\widehat{u}(n_2,t)\right)\\
- 2i\sum_{n_1+n_2+n_3=n}\frac{1}{n_2}\widehat{u}(n_1,t)\widehat{u}(n_2,t)\widehat{u}(n_3,t).
\end{multline*}

When $s\leq -\f{1}{2}$, nonlinear smoothing properties obtained by a standard normal form transformation are not sufficient.  To overcome this difficulty, we will modify the normal form transform operator by restricting size of its modulation weight $\lan{\tau - n^3}$.  Formally, this modulation-restricted normal form achieves the same effect as the standard one, in that it removes an undesirable nonlinear term from the RHS.  In contrast to the standard one, a modulation-restricted normal form can leave intact an \textit{acceptable} (in the sense of boundedness) portion of the nonlinearity instead of removing it completely.  This restriction prevents giving rise to undesirable modulation/frequency interactions in higher order terms, which would be unbounded otherwise.\\

As introduced in \cite{TT, NTT}, we adapt the use of $X^{s,b}$-space with initial-data-dependent Fourier weight.  We will take advantage of $\ell^2$-decoupling theory \cite{schippa} to preserve the $L^6_{t,x}$ embedding for the modified Bourgain space.\\

The outline of the paper is as follows: In Section~\ref{Section: Notation}, we outline common notations that will be persistent throughout the manuscript. In Section~\ref{Section: Function Spaces}, we define auxiliary spaces and establish a few linear estimates.  In Section~\ref{Section: Problem Set-up}, we apply a modulation-restricted normal form transformation and simplify the resulting equation.  Section~\ref{section: estimates} contains all nonlinear estimates to be used for the proof of Theorem \ref{Theorem: 1}. Lastly, we prove Theorem \ref{Theorem: 1} in Section \ref{Section: Proof Thm1}.

\section{Notations}\label{Section: Notation}
 
 For any $u\in C^{\infty}(\bt)$, we denote the Fourier coefficients by 
 \[
 u_n := \cf_x[u] = \int_{\bt} u(x) e^{-inx} \,dx.
 \]
 Similarly, for $u\in \cs(\br)$ we denote the time Fourier transform by 
 \[
 \wh{u}(\tau):= \cf_t[u](\tau) = \int_{\br} u(t) e^{-it \tau}\, dt,
 \]
 as well as the space-time Fourier transform 
 \[
 \wt{u}(\tau,n) := \cf_{t,x}[u](\tau,n) =\int_{\mathbb{R}} \int_{\bt} u(t,x) e^{inx +i \tau t} \,dx\, dt.
 \] 
 We will often be required to invert these, denoted by $\cf^{-1}$, which can be interpreted with respect to the variables indicated in the subscript, and it can represent either the Fourier series or the inverse Fourier transform. Lastly, Denote the inhomogeneous operator $\lan{\cdot} := (1+ |\cdot|^2)^{\f{1}{2}}$, so that we may define the Fourier multiplier $\nab^s u := \cf^{-1}_n [\lan{n}^s u_n]$, for any $s\in \br$\\

For two positive quantities, $A\lesssim B$ means that there exists $C>0$ such that $A \leq C\, B$.  Any dependence of the implicit constant $C$ will be listed below $\lesssim$.  For instance, $A \lesssim_{\n{f}{L^2}} B$ means that the implicit constant depends on $\n{f}{L^2}$. Relation $\gtrsim$ is defined analogously.  If $A\lesssim B$ and $A\gtrsim B$, then we denote $A\sim B$.  Relation $\ll$ (respectively $\gg$) is defined the same way as $\lesssim$ (respectively $\gtrsim$), except here the implicit constant $C$ is assumed to be much larger than 1.\\

For any number $c$, we denote $c+$ (respectively, $c-$) to represent $c+\ve$ (respectively, $c-\ve$) where $\ve>0$ can be chosen to be arbitrarily small.\\

Before proceeding further, let us make a preliminary change in the format of the equation. By substituting $u\mapsto \nab^{-s} u$, we can place \eqref{eq:kdv original} into the $L^2(\mathbb{T})$ space: 
\begin{equation}\label{eq:kdv}
\left|\begin{array}{l} u_t + u_{xxx}= \nab^{-|s|} \p_x[ \nab^{|s|} u \, \nab^{|s|} u]) =: \cn(u,u)\\ \left.u\right|_{t=0} = \nab^{-s} u_0 =: f \in L^2(\bt). \end{array}\right.
\end{equation}
 This will enable us to work in $L^2$-based spaces, which reduces clutter.  Henceforth, we will use positive values for $s$ by denoting $s\mapsto |s|$ for notational conciseness.  In particular, this means that the range of $s$ given in Theorem~\ref{Theorem: 1} is $ \f{1}{2} \leq s < \f{2}{3}$. \\

Given a bilinear Fourier multiplier $T_\sigma (\cdot, \cdot)$ with symbol $\sigma$, we denote $T_\sigma(\chi \cdot, \cdot)$ such that for any smooth functions $u$ and $v$,
\begin{equation}\label{eq:chi}
\cf_{t,x} [T_{\sigma} (\chi u, v)](\tau,n) =\underset{\substack{\tau_1+\tau_2=\tau\\n_1+n_2=n}}{\int\sum} \sigma  \wt{\chi}(n_1,\tau_1;n_2) \wt{u}(\tau_1,n_1)\wt{u}(\tau_2, n_2) \, d\tau_1
\end{equation}
where $\wt{\chi}(n_1,\tau; n_2)$ is a smooth function satisfying
\[
\wt{\chi}(n,\tau;k) = \begin{cases} 1 &\text{ if } \langle \tau - n^3-\phi_n\rangle \lesssim 3(n+k)nk\\ 0 &\text{ if } \langle \tau - n^3-\phi_n\rangle \gg 3(n+k)nk\end{cases}
\]
where
\[
\phi_n:= \frac{2}{3} \frac{\langle n\rangle^{2s}}{n} |f_n|^2,
\]
and $\{f_n\}_{n\in \mathbb{Z}\setminus\{0\}}$ is the Fourier coefficient of the initial data in \eqref{eq:kdv}.  When the $\chi$ operator is applied to the second entry, it takes the form:
\[
\cf_{t,x} [T_{\sigma} ( u, \chi v)](\tau,n) =\underset{\substack{\tau_1+\tau_2=\tau\\n_1+n_2=n}}{\int\sum} \sigma  \wt{u}(\tau_1,n_1)\wt{\chi}(n_2,\tau_2;n_1) \wt{u}(\tau_2, n_2) \, d\tau_1.
\]

Utilizing this, we will often decompose a bilinear operator as $T_\sigma = T_\sigma^\ell + T_{\sigma}^h$ where 
\[T_{\sigma}^{\ell}(u,v) := T_{\sigma}(\chi u, \chi v).\]
We may then rewrite $T^h_\sigma$ as:
\begin{align} 
T_{\sigma}^{h}(u,v) &= T_{\sigma}((1-\chi) u, v)  + T_{\sigma}( u, (1-\chi) v) - T_{\sigma}((1-\chi) u, (1-\chi) v) \nonumber\\
            &= T_{\sigma}((1-\chi) u, \chi v)  + T_{\sigma}( u, (1-\chi) v) \nonumber\\
            &= T_{\sigma}((1-\chi) u, \chi v)  + T_{\sigma}(\chi u, (1-\chi) v) + T_{\sigma} ((1-\chi) u, (1-\chi) v) \nonumber\\
            &=: T_{\sigma}^{h\ell}(u,v) + T_{\sigma}^{\ell h}(u,v)+ T_{\sigma}^{h h}(u,v).\label{eq:ell}
\end{align}
Note that in the ``double-superscript'' notation, the first letter denotes the `high' versus `low' restriction of the first argument, and the second letter denotes the same for the second argument.  For the single superscript case, $T_{\sigma}^\ell = T_{\sigma}^{\ell \ell}$.  But the same is not true for the superscript $h$. Additionally, we remark that the application of $\chi$ in the above is only heuristically as an operator of the form $T_\sigma(\chi u, \chi v)$; $\chi$ depends on the frequencies of both entries, so this is is really of the form $T_{\sigma\tilde{\chi_1}\tilde{\chi_2}}(u,v)$ for some $\chi_1$ and $\chi_2$.

\section{Functional spaces}\label{Section: Function Spaces}
We now define the spaces we'll be working in.  Let $L_n = n^3 + \phi_n$ and $W_t$ be the linear propagator associated to the linear operator $L_n$, i.e. $\widehat{W_t g}(n)$ is the solution to:
\begin{equation}\label{Equation: Equation Dependent Linear Equation}
\begin{cases}
    (\partial_t -i L_n)u_n = 0\\
    u_n(0) = g_n\in \ell^2_n,
\end{cases}
\end{equation}
For a dyadic index $L\geq 1$, we denote the modulation localization of $v$ as $v_L$ where 
\[
v_L := \cf^{-1}_{\tau, n}\left[ \chi_{\lan{\tau - n^3 -\phi_n}\sim L} \wt{v}(\tau,n)\right]
\]
Define localization operators $\chi_A$ where $A$ is a condition that must be met such as $\chi_{L \ll n}$, etc, and define the smooth weights:
\begin{align}
    w_Y(n,L) &:= 
        \begin{cases}
            L^{1/2+} & L\ll |n|\\
            L^{1/3+}\langle n\rangle^{1/3-} & L\gtrsim |n|
        \end{cases}\label{Equation: Y weight}\\
    w_Z(n,L) &:=
        \begin{cases}
            L^{1/2+} & L\ll |n|\\
            L^{2/3+} & L\gtrsim |n|.
        \end{cases}\label{Equation: Z weight}
\end{align}
Then, the initial-data-dependent Bourgain space is defined as:
\[
\|v\|_{X^{s,b}} := \n{\lan{n}^s \langle \tau - n^3- \phi_n\rangle^{b}\wt{v}(\tau, n)}{\ell^2_nL^2_\tau}.
\]
The first space we define is our base space, and has norm is given by:
\begin{multline*}
\|v_L\|_{Y} := \|w_Y(n,L)\tilde{v}_L (\tau,n)\|_{L^2_{n,\tau}} + \|\wt{v}_L(\tau,n)\|_{\ell^2_nL^1_\tau}\\
\sim \left\|\left(L^{1/2+}\chi_{L\ll|n|} + \langle n\rangle^{1/3-}L^{1/3+}\chi_{L\gtrsim |n|}\right)\wt{v}_L(\tau, n)\right\|_{\ell^2_nL^2_\tau}+\|\wt{v}_L(\tau,n)\|_{\ell^2_nL^1_\tau}.
\end{multline*}

Utilizing these, we can define a secondary space, denoted $Z$, which is smoother in time than $Y$:
\begin{align*}
\|v_L\|_Z &:= \|w_Z(n,L)\tilde{v}_L(\tau,n)\|_{L^2_{\tau,n}} \sim \n{ \pr{L^{\f{1}{2}+}\chi_{L \ll n}  + L^{\f{2}{3}+}\chi_{L \gtrsim n}}\wt{v}_L(\tau, n)}{L^2_{\tau,n}}.
\end{align*}
A standard square-sum is then utilized to combine the terms, so that
\[
    \|v\|_{Y}^2 = \sum_{L} \|v_{L}\|_{Y}^2,\qquad\|v\|_{Z}^2 = \sum_{L} = \|v_{L}\|_{Z}^2.
\]

Note that we have the chain of inclusions given by $Z \hookrightarrow  Y \hookrightarrow X^{0,1/3+}$.  
Additionally, the low modulation adaptation and the high modulation smoothing are chosen so that $u\in Y$ implies $\nab^{0-}u\in L^6_{x,t}$.  This fact will be shown later.

\begin{remark}
    Similar spaces have been defined before, for instance \cite{MR3013402} in connection to the Kawahara equation.  Also, in  \cite{shin2024nonlinearsmoothingperiodicdispersion},  Shin proves well-posedness utilizing Differentiation-by-parts in an $X^{s,b}$ framework for the dispersion generalized Benjamin-Ono equation, similar to the above. However, they utilize an idea of Tao \& Bejenaru \cite{MR2204680} and Kato \cite{MR3013402} in modifying the derivative weight when the modulation is large. This enables them to estimate purely with $b = \frac12$, at the cost of some high modulation derivatives. This idea fails in our framework, as the key problem lies with a $H\times H\to L$ frequency interaction as opposed to a $H\times L\to L$ interaction. 
\end{remark}

We can obtain a standard $X^{s,b}$-type smoothing estimate:
\begin{lemma}
\label{le:19}
Let $\eta \in \cs_t(\br)$.  For any $b>1/2$ and $s\in \br$,
$$
\n{\eta(t)\int_0^t W_{t-s} F_n(s) ds}{Z} \lesssim_{\eta,b} \n{\cf^{-1}_{n,\tau} [\lan{\tau - n^3 - \phi_n}^{-1} \wt{F}_n(\tau)]}{Z}.
$$
\end{lemma}

We will denote the norm on the RHS of Lemma~\ref{le:19} by $Z^*$:
\begin{multline*}
\n{v_L}{Z^*} := \n{\cf^{-1}_{n,\tau} [\lan{\tau - n^3 - \phi_n}^{-1} \wt{v}_L(\tau,n)]}{Z}\\
\sim \left\|\left(L^{-1/2+}\chi_{L\ll |n|} + L^{-1/3+}\chi_{L\gtrsim |n|}\right)\tilde{v}_L(\tau, n)\right\|_{L^2_{n,\tau}}.
\end{multline*}

We now record several useful embeddings that will enable us to estimate multilinear operators of functions in $Y$ and $Z$.

\begin{lemma}\label{le:l4}
$Z \hookrightarrow Y \hookrightarrow X^{0,\f{1}{3}+} \hookrightarrow L^4_{t,x}$
\end{lemma}

Additionally, the $L^6_{t,x}$ embedding can be recovered through the $\ell^2$-decoupling framework.
\begin{lemma}\label{le:l6}
$X^{0+,\f{1}{2}+} \hookrightarrow L^6_{t,x}$
\end{lemma}

\begin{proof}
  This  utilizes the work of Schippa \cite{schippa}. Recall that $\phi_n$ is given as
$\phi_n = \frac{\langle n\rangle ^{2s}}{n}|\widehat{f}_n|^2$.  Let $\eta\in\mathcal{S}$ have support contained in $[-1/4,1/4]$ and satisfy $\int \eta\,dx = 1$, and let $g(x)$ be defined by
\[
g(x) =
\begin{cases}
    |\widehat{f}_n|^2 & \mbox{if }x\in[n-\frac{1}{4}, n+\frac{1}{4}] \text{ for any } n \in \mathbb{Z}\\
    0 & \mbox{else}.
\end{cases}
\]
Then, for $\xi\ne 0$ we define $G(\xi) = \frac{\langle \xi\rangle^{2s}}{\xi}\eta*g(\xi)$, so that for any $n\in\mathbb{Z}\setminus \{0\}$ we have $G(n) = \phi_n.$ We remark that there is no issue at $0$ because $f(t)\equiv 0$ in a neighborhood of $0$, but we do this because we will work away from $0$ anyway.\\

With this setup, we find that $\vp(\xi) = \xi^3+G(\xi)$ is a $C^2$ phase function which behaves exactly the same way as $\xi^3$ for $|\xi| \gg_{g,\eta} 1$. In the language of \cite{schippa}, for all $\xi\gg 1$ we have that $\vp$ satisfies $\mathcal{E}^{0}(\psi)$ where $\psi(N) = 6N$ for $N\gg 1$. It then follows by \cite[Proposition 1]{schippa} that
\[
\left\|\sum_{|n|\sim N}e^{it(n^3+\phi_n)+inx} g_n\right\|_{L^6(\mathbb{T}\times\mathbb{T})} \lesssim N^{\varepsilon}\|P_N g_n\|_{\ell^2_n}.
\]
Applying the square-function estimate after taking $2\ve$ derivative, we obtain
\[
\n{J^{0-}_x \cf^{-1}_n e^{it(n^3 + \phi_n)} g_n}{L^6_{t,x}} \lesssim \n{g}{L^2}
\]
To show the embedding, note that the implicit constant of this proof does not change when we define $\vp_{\tau_0} (\xi) = \xi^3 + G(\xi) + \tau_0$ for any $\tau_0$.  Applying \cite[Lemma 2.9]{tao:text}, we have the spatial embedding.
\end{proof}

\section{Problem set-up}\label{Section: Problem Set-up}
We now proceed to reduce the equation \eqref{eq:kdv} to the system of equations we work with, given by $u = T^\ell(u,u) + v$ with $v$ satisfying \eqref{eq:vn}. Towards that goal, we note that
\[
 \cf_x[\mathcal{N}(u,u)] = \sum_{n=n_1+n_2}\frac{ n \langle n_1\rangle^s\langle n_2\rangle^s}{ \lan{n}^s}\widehat{u}(n_1)\widehat{u}(n_2),
\]
and introduce the $\chi$ operators through the decomposition
\begin{equation}\label{definition: Nh}
\mathcal{N}(u,u) = \mathcal{N}^\ell(u,u) + \mathcal{N}^h(u,u)  
\end{equation}
where $\mathcal{N}^\ell(u,u) = \mathcal{N}(\chi u, \chi u)$ as defined in \eqref{eq:chi}. Next, define the standard KdV normal form operator (see \cite{titi,erdougan2013global,Oh}) given by:
\begin{equation}
\cf_{x}[T(f,g)](n)
:= \sum_{n=n_1+n_2}\frac{\langle n_1\rangle^s\langle n_2\rangle^s}{{ 3 \lan{n}^s }n_1n_2}f_{n_1}g_{n_2},\label{Definition: Tl definition}
\end{equation}
for any $f,g\in C^\infty_{x}(\bt)$. The modulation-restricted normal form operator is defined by
\[
T^\ell(u,v) = T(\chi u, \chi v)
\]
where $u$ and $v$ are functions of $(t,x)$.  In this case, $T^\ell$ formally satisfies
\[
(\p_t + \p_x^3) T^\ell(u,u) = 2 T^\ell (\mathcal{N}(u,u), u) + \mathcal{N}^\ell(u,u).
\]
We now let $u = T^\ell(u,u) + v$, and utilize the above display to obtain the equation for $v$ as
\begin{equation}\label{eq:transform1}
\left|\begin{array}{l} (\p_t + \p_x^3) v = 2T^{\ell} (\mathcal{N}(u,u),u) + \mathcal{N}^h (u,u)\\
v(0) = f - T^{\ell}(f,f).\end{array}\right.
\end{equation}
\begin{remark}
    The object $T^\ell (f,f)$ is not well-defined as $T^\ell$ requires the full space-time Fourier transform of $u$, and is not writable strictly in terms of $f$. In Theorem~\ref{le:tell1}, we will see that, for any $u(t)\in C^0_tL^2_x$, $T^\ell (u,u) \in C^{0}_t H^s_x$ for some $s>0$ so that we can define  $T^\ell(f,f) :=  \left.T^{\ell} (u,u)\right|_{t=0}$. Other locality-related issues will be addressed in Section \ref{Section: Proof Thm1}.
\end{remark} 

Consider the trilinear term on the RHS, $T^\ell (\mathcal{N}(u,u),u)$, which, on the Fourier side, can be written as:
\begin{multline}
3 \cf_{t,x}\left[T^\ell (\mathcal{N}(u,u),u)\right] (\tau,n) \\
=\underset{\tiny \begin{array}{c}n_1+n_2\neq 0,
\quad n_j \neq 0\\ n_1+n_2+n_3 = n\\ \tau_1+\tau_2+\tau_3=\tau\end{array}}{\int\sum} \frac{\langle
n_1\rangle^s \langle n_2\rangle^s \langle n_3\rangle^{s}}{i n_3\langle
n\rangle^s} \chi(n_1+n_2,\tau_1+\tau_2;n_3) \wt{u}(\tau_1, n_1) \\
\times\wt{u}(\tau_2, n_2) \chi(n_3,\tau_3;n_1+n_2)\wt{u}(\tau_3,n_3).
\end{multline}

Trilinear resonance occurs when $(n_1+n_2)(n_2+n_3)(n_3+n_1) =  0$, so we let $\mathcal{R}^\ell(u,u,u)$ denote the restriction of $3T^\ell (\mathcal{N}(u,u),u)$ to this condition and define $\mathcal{M}$ so that 
\begin{equation}\label{Definition: M definition}
3T^\ell (\mathcal{N}(u,u),u) =\mathcal{R}^\ell(u,u,u)+ \mathcal{M} (u,u,u).
\end{equation}
Let $\mathcal{R}(u,u,u)$ denote the trilinear resonant term rising from $T(\cn(u,u),u)$ (that is, without $\chi$ operators).  We denote
\begin{equation}\label{definition: Mod Rest. Resonant}
\mathcal{R}^h(u,u,u) := \mathcal{R}(u,u,u) - \mathcal{R}^\ell(u,u,u).
\end{equation}
Under this convention, we find that that our equation for $v$ reads:
\[
\left|\begin{array}{l} (\p_t + \p_x^3) v = \mathcal{N}^h (u,u) + \f{2}{3} \left[\mathcal{M} (u,u,u) + \mathcal{R}(u,u,u) - \mathcal{R}^h(u,u,u)\right]\\
v(0) = f - T^\ell (f,f). \end{array}\right.
\]
From \cite{Oh}, we know that the Fourier coefficients of $\mathcal{R}(u,u,u)$ reduce to
\[
\mathcal{R}_n(u,u,u) = i \f{\lan{n}^{2s}}{n} |u_n|^2 u_n,
\]
and that the IVP: $(\p_t + \p_x^3) r = \f{2}{3}\mathcal{R}(u,u,u)$ and $r(0) = f$ yields  the explicit solution
\[
r(t,x) = \sum_{n\neq 0} f_n \exp \left( \f{2}{3} i \f{\lan{n}^{2s}}{n} | f_n|^2 \,t\right) e^{i(nx + n^3 t)} = \sum_{n\neq 0} f_n e^{i(\phi_n+ n^3)t} =: W_t f.
\]

\subsection{A Cancellation Structure} We now seek to leverage a cancellation structure utilized first, to our knowlege, by Takoaka \& Tsutsumi \cite{TT}. Fundamentally, this is the observation that we may replace the trilinear resonant interaction with an integral of the \textit{nonresonant} interaction. This means that our Duhamel formulation will have two time integrals, which we handle in Lemma \ref{le:vkn}.\\

To that end, we temporarily introduce the variable $w$ by removing the linear evolution from $v$, via $v = W_tf + w$, which translates to 
\[
u =  W_tf + T^{\ell}(u,u) + w.
\]
In this section, we will denote $r :=  W_tf$ and $h := T^\ell(u,u)$, so that $u = r + h+ w$.  The equation for $w$ can be written as
\begin{equation}\label{eq:transform2}
    \left|\begin{array}{l} (\p_t + \p_x^3) w =  \frac{2}{3} \mathcal{R}^*(u) + \mathcal{N}^h (u,u) +  \f{2}{3} \mathcal{M}(u,u,u) + \f{2}{3} \mathcal{R}^h(u,u,u)\\
w(0) =-T^\ell (f,f), \end{array}\right.
\end{equation}
where  $\mathcal{R}^*(u) = \mathcal{R}(u,u,u) - \mathcal{R}(r,r,r)$. \\

Since the resonant solution operator $W_t(\cdot)$ is unitary, we know that  $r$ cannot be smoother than the initial data, $f\in L^2$.  On the other hand, as will be observed in Lemma~\ref{le:tell1}, $h$ receives smoothing benefit from the normal form operator $T^\ell$, while the remainder term $w$ is comparable to $r$ (and have additional temporal regularity). We now perform a substitution in order to obtain estimates to close our proof, we substitute the relationship for $u$ into $\mathcal{R}^*$
\begin{equation}\label{Equation: Resonant Expansion}
\mathcal{R}^*(u,u, u) = \mathcal{R}(r+h+w, r+h+w, r+h+w) - \mathcal{R}(r,r,r),
\end{equation}
which will then lack the roughest term, $\mathcal{R}(r,r,r)$. We observe that any $\rr(\cdot, \cdot, \cdot)$ term that contains at least one entry of $h$ is sufficiently smooth as will be observed in Lemma~\ref{le:r}, so we restrict our attention to the terms of \eqref{Equation: Resonant Expansion} that only involve $r$ or $w$. The terms in the above display that are linear in $w$ and quadratic in $r$ satisfy
\[
\mathcal{R}_n(r, r,w) +\mathcal{R}_n (w, r,r)+ \mathcal{R}_n(r,w, r) = i\phi_n w_n + 2i\f{\lan{n}^{2s}}{n} r_n \cdot \textnormal{Re}[ \overline{r_n} w_n ],
\] 
while those that are quadratic in $w$ and linear in $r$ satisfy
\[
\mathcal{R}_n(r, w,w) +\mathcal{R}_n (w, r,w)+ \mathcal{R}_n(w,w, r) = i  \f{\lan{n}^{2s}}{n}r_n |w_n|^2 + 2i\f{\lan{n}^{2s}}{n} w_n \cdot \textnormal{Re}[ \overline{r_n} w_n ].
\] 
When combined with the cubic term in $w$, we obtain the expression
\begin{multline*}
\mathcal{R}_n(r, r,w) +\mathcal{R}_n (w, r,r)+ \mathcal{R}_n(r,w, r)  + \mathcal{R}_n(r, w,w)\\
+\mathcal{R}_n (w, r,w)+ \mathcal{R}_n(w,w, r) + \mathcal{R}_n(w, w,w)\\
= i \phi_n w_n + 2i\f{\lan{n}^{2s}}{n} (r_n+w_n) \left( \textnormal{Re}[ \overline{r_n} w_n ]  + \f{1}{2}  |w_n|^2\right) .
\end{multline*}

Denoting $L_n = n^3+ \phi_n$, the equation for $w_n$ can be written as
\begin{equation}\label{eq:w2}
\pr{\f{d}{dt} - iL_n } w_n =  2i\f{\lan{n}^{2s}}{n} (r_n+w_n) \left( \textnormal{Re}[ \overline{r_n} w_n ]  + \f{1}{2}  |w_n|^2\right)  +  \mathcal{NR}_n 
\end{equation}
where $\mathcal{NR}$ captures all remaining terms, which includes:
\begin{equation}\label{eq:nr}
\rr(u,u,u) - \rr(r+w,r+w,r+w); \quad \rr^h(u,u,u); \quad \cn^h(u,u); \quad \mathcal{M}(u,u,u).
\end{equation}
We observe that the first expression in \eqref{eq:nr} must include a factor of $h = T^\ell(u,u)$ which is smooth as noted before.  The first term on the RHS of \eqref{eq:w2} is the most difficult to handle, since we start off with a loss of $2s-1$ derivatives, which must be recuperated somehow.  Hence, we establish the aforementioned cancellation property to overcome this derivative loss in the following Lemma.

\begin{lemma}\label{le:pn}
    If $\{w_n\}_{n\in \mathbb{Z}\setminus \{0\}}$ is a smooth function satisfying \eqref{eq:w2}, we must have
    \[
 \textnormal{Re}[ \overline{r_n} w_n ]  + \f{1}{2}  |w_n|^2 =  \int_0^t  \textnormal{Re}[\overline{r_n+w_n}\mathcal{NR}_n](s) \, ds + \textnormal{Re}[\overline{f_n} w_n(0)] + \f{1}{2} \left|w_n(0) \right|^2
    \]
\end{lemma}

\begin{proof}
We seek to utilize the cancellation structure within $P_n :=\textnormal{Re}[ \overline{r_n} w_n ]  + \f{1}{2}  |w_n|^2$, where we note that $P_n$ is real valued.  Multiplying both sides of \eqref{eq:w2} by $e^{-itL_n}$ where $L_n = n^3 + \phi_n$, we write 
\[
\f{d}{dt}( e^{-it L_n} w_n) =  2ie^{-it L_n}\f{\lan{n}^{2s}}{n} r_n P_n + 2ie^{-it L}\f{\lan{n}^{2s}}{n} w_n  P_n  + e^{-it L_n}\mathcal{NR}_n.
\]
Further multiplying by $\overline{g_n}$ and writing $e^{-itL_n} \overline{g_n} = \overline{r_n}$, we observe:
\[
\f{d}{dt}( \overline{r_n} w_n) =  2i\f{\lan{n}^{2s}}{n} |r_n|^2 P_n + 2i\f{\lan{n}^{2s}}{n}  \overline{r_n} w_n  P_n  + \overline{r_n}\mathcal{NR}_n,
\]
where we may now take the real part
\[
\f{d}{dt}\textnormal{Re} ( \overline{r_n} w_n) =    -2\f{\lan{n}^{2s}}{n} \IM[ \overline{r_n} w_n]  P_n  + \textnormal{Re} [\overline{r_n}\mathcal{NR}_n].
\]
We now consider the $\f{1}{2}|w_n|^2$ term.  Noting that  $\f{d}{dt}|w_n|^2 =  2\textnormal{Re} (\overline{w_n} \f{d}{dt} w_n)$, we multiply \eqref{eq:w2} by $\overline{w_n}$
\[
\overline{w_n}\f{d}{dt} w_n  - iL  |w_n|^2 =  2i\f{\lan{n}^{2s}}{n} \overline{w_n} r_n P + 2i\f{\lan{n}^{2s}}{n} |w_n|^2 P + \overline{w_n} \mathcal{NR}_n ,
\]
and again take the real part:
\[
\textnormal{Re} [\overline{w_n}\f{d}{dt} w_n]  =  -2\f{\lan{n}^{2s}}{n} \IM [\overline{w_n} r_n] P_n +  \textnormal{Re} [\overline{w_n} \mathcal{NR}_n ].
\]
Finally, noting that $\IM (\overline{r_n} w_n)  = - \IM (\overline{w_n} r_n)$, we observe that $P_n$ satisfies
\[
\f{d}{dt} P_n = \f{d}{dt} \textnormal{Re}[ \overline{r_n} w_n] + \textnormal{Re}[ \overline{w_n} \f{d}{dt} w_n] =  \textnormal{Re}[\overline{r_n+w_n}\mathcal{NR}_n],
\]
which yields the desired result by integrating in time.
\end{proof}

Applying Lemma~\ref{le:pn} to \eqref{eq:w2}, we have
\begin{align*}\label{eq:w3}
\pr{\f{d}{dt} - iL_n } (r_n + w_n) &=  2i\f{\lan{n}^{2s}}{n} (r_n+w_n) \int_0^t  \textnormal{Re}[\overline{r_n+w_n}\mathcal{NR}_n](s) \, ds \\
&+2i\f{\lan{n}^{2s}}{n} (r_n+w_n)\left( \text{Re}(\overline{f_n}w_n(0)) + \f{1}{2} \left|w_n(0) \right|^2\right)\\
&+  \mathcal{NR}_n
\end{align*}

Where the LHS was modified using the identity  $\pr{\f{d}{dt}- i L_n} r_n = 0$.  Rewriting this expression purely in terms of $v$ we obtain our equation for $v$:
\begin{equation}\label{eq:vn}
    \left|\begin{array}{l}
        (\partial_t - iL_n)v_n = 2i\frac{\langle n\rangle^{2s-1}}{n}v_n\int_0^t\mathrm{Re}[\overline{v}_n\mathcal{NR}_n](s)\,ds + \mathrm{Re}[\overline{f}_nw_n](0)+ \frac{1}{2}|w_n(0)|^2 + \mathcal{NR}_n\\
        v_n(0) = f_n - [T^\ell(f,f)]_n,
    \end{array}\right.
\end{equation}
where $w_n(0) = - [T^\ell (f,f)]_n$.
\section{Estimates}\label{section: estimates}

\subsection{Normal form estimate}
We first demonstrate estimates for $h = T^\ell$ that will be useful throughout the manuscript. These estimates not only allow us to define $T^\ell(f,f)$, but also allow use to treat $h$ as a single term in many of the estimates to come.
\begin{lemma}\label{le:tell1}
Let $h = T^{\ell}(u,u)$ where $u \in Y$.  If $\f{1}{2} \leq s < \f{2}{3}$,
\[ 
\n{h}{C^0_t H^{\f{3}{2}-s-}_x} + \n{h}{L^2_t H^{1-}_x} + \n{h}{X^{1/3-,1/3+}} + \n{\langle n\rangle^{\frac32-s-}\tilde{h}(\tau,n)}{\ell^2_n\ell^1_\tau} \lesssim \n{u}{Y}^2.
\]

\end{lemma}
\begin{proof}
We denote the time-space frequencies of $h$ by $(\tau,n)$, while internal frequencies for the two $u$'s inside are denoted $(\tau_j, n_j)$ for $j=1,2$, where $\sum_j (\tau_j, n_j) = (\tau,n)$.  As we will see below, the first two norms on the LHS only require estimating the size of the spatial-frequency symbol, along with the embedding $Y \hookrightarrow L^\infty_t L^2_x \cap L^4_{t,x}$.\\

The cornerstone of all of the estimates lies in the symbol estimate
 \begin{equation}\label{Equation: Nl symbol for v lemma}
    \sigma(T^\ell) = \frac{\langle n_1\rangle^s\langle n_2\rangle^s}{\langle n\rangle^s n_1n_2} \lesssim \f{\lan{n_1}^{s-1} \lan{n_2}^{s-1}}{\lan{n}^s},
    \end{equation}
with which we will show each estimate in turn. To show $h \in C^0_t H^{\f{3}{2} - s-}$, note:
 \[
\lan{n}^{\f{3}{2} - s-} \si(T^\ell) \lesssim \lan{n}^{\f{3}{2} - 2s } \lan{n_1}^{s-1}  \lan{n_2}^{s-1} \lesssim \lan{n_{\max}}^{\f{1}{2} - s-} \lan{n_{\min}}^{s-1} \lesssim \lan{n_{\min}}^{-\f{1}{2}-},
\]
where we have assumed that $s \geq \f{1}{2}$. Since RHS is bounded by $\lan{n_{\min}}^{\f{1}{2} -}$, we may apply Sobolev embedding,
\[
\n{h}{L^\infty_t H^{\f{3}{2} - s -}_x} \lesssim \n{u \nab^{-\f{1}{2}-} u}{L^\infty_t L^2_x} \lesssim \n{u}{L^\infty_t L^2_x} \n{\nab^{-\f{1}{2}-} u}{L^\infty_{t,x}}  \lesssim \n{ u }{L^\infty_t L_x^2}^2\lesssim  \n{ u }{Y}^2.
\]
This similarly handles the $\ell^2_nL^1_\tau$ portion, where an application of Young's replaces the $L^\infty$ estimate.\\

To show that $h \in L^2_t H^{1-}_x$, note that for any $s \leq 1$:
\[
\lan{n}^{1-} \si(T^\ell) \lesssim \lan{n}^{1-s } \lan{n_1}^{s-1}  \lan{n_2}^{s-1} \lesssim 1,
\]
and hence an application of H\"older's inequality yields our desired bound,
\[
\n{h}{L^2_t H^{1-}_x} \lesssim \n{u^2}{L^2_{t,x}} \lesssim \n{u}{L^4_{t,x}}^2 \lesssim \n{u}{Y}^2.
\]

The final bound is  $\n{h}{X^{\frac13-,\f{1}{3}+}}$. Consider space-time frequency weights in the form of modulational weights: $L \sim \lan{\tau -n^3 -\phi_n}$ and $L_j \sim \lan{\tau_j - n_j^3 - \phi_{n_j}}$ for $j=1,2$.  We denote $L_{\max} = \max \{L, L_1, L_2\}$.  Furthermore, we denote
\[
H_2 = (\tau - n^3) - (\tau_1 - n_1)^2 - (\tau_2 - n_2)^2 = n n_1 n_2,
\]
where we neglect $\phi_n$'s as $\phi_n \lesssim_f \lan{n}^{2s-1} \ll \lan{n}$ for any $s<1$.\\

We need to consider two cases: $L \ll H_2$ and $L \sim H_2$.   Note that, due to the restriction on $T^\ell$, we must have $L_{\max} \sim H_2$. If  $L \ll H_2$, then either $L_1 = L_{\max}$ or $L_2 = L_{\max}$.  Say $L_1 = L_{\max}\sim H_2$.  Here, we only need that $L_1 \gtrsim L$:
\[
\f{\lan{n}^{\f{1}{3}-} L^{\f{1}{3}+} \si (T^\ell)}{L_1^{\f{1}{3}+}} \lesssim \lan{n}^{\f{1}{3} - s - } \lan{n_1}^{s- 1} \lan{n_2}^{s - 1} \lesssim  \lan{n_{\min}}^{-\f{1}{2}-}.
\]
We can gain a lot more derivatives in this case, but we will not need this in our subsequent arguments. Applying Sobolev embedding,
\[
\n{\nab^{\f{1}{3}-} h}{X^{0,\f{1}{3}+}} \lesssim \n{\pr{L_1^{\f{1}{3}+} u} \pr{ \nab^{-\f{1}{2}-} u}}{L^2_{t,x}} \lesssim \n{L_1^{\f{1}{3}+} u}{L^2_{t,x}} \n{\nab^{-\f{1}{2}-} u}{L^\infty_{t,x}}  \lesssim \n{ u }{X^{0,\f{1}{3}+}} \n{u}{L^\infty_t L^2_x}
\]
which is bounded by $\n{u}{Y}^2$.  Finally, if $L \sim H_2$, note:
\[
\lan{n}^{\f{1}{3} - 2\ve} L^{\f{1}{3}+\ve} \si (T^\ell) \sim \lan{n}^{\f{2}{3}  -s - \ve} \lan{n_1}^{s- \f{2}{3}+ \ve} \lan{n_2}^{s - \f{2}{3} + \ve}  \lesssim 1
\]
as long as $s < \f{2}{3}$.  Then, $\ds \n{\nab^{\f{1}{3}-} h}{X^{0,\f{1}{3}+}}\lesssim \n{u^2}{L^2_{t,x}}$, which can be closed using the $L^4_{t,x}$ embedding as before. 
\end{proof}
\begin{remark}
Note that, if $v\in Z$, 
\[
\n{\nab^{\f{1}{3}} \chi_{L\gtrsim n} v}{X^{0,\f{1}{3}+}} \lesssim \n{L^{\f{1}{3}} \chi_{L\gtrsim n} v}{X^{0,\f{1}{3}}} \lesssim \n{ L^{\f{2}{3}+} \chi_{L\gtrsim n} v}{L^2_{t,x}} \lesssim \n{v}{Z}.
\]
\end{remark}

 As previously mentioned, $\nab^{0-}_xu\in L^6_{x,t}$ for $u\in Y$. Indeed, the low modulation is simply the Strichartz estimate, while the high modulation follows from $X^{1/3+, 1/3+}\hookrightarrow L^6_{x,t}$, which is obtained from interpolation of two embeddings $X^{1/2+,0}\hookrightarrow L^2_t L^\infty_x$ and $X^{0,1/2+}\hookrightarrow L^\infty_t L^2_x$.\\

 While we always will need $\ve$-derivative loss to place $u$ in $L^6_{t,x}$, it will be clear from all contexts that this derivative loss can be recovered by the symbol estimate.  We will omit this detail for the sake of brevity.\\

\subsection{Estimation of Integral Term}
Due to the derivative loss in \eqref{eq:vn}, we will end up needing to estimate the nonlinearity in a separate space. Let $w_{\tilde{X}}$ be a smooth weight defined by
\begin{equation}
    w_{\tilde{X}}(n,L) := 
        \begin{cases}
            L^{-1/2+}\langle n\rangle^{2s-1+} & L\ll |n|\\
            L^{-1/3+} & L\gtrsim |n|
        \end{cases}\label{Equation: X weight},
\end{equation}
and the space this defines by:
\begin{multline*}
\|v_L\|_{\tilde{X}} := \|w_{\tilde{X}}(n,L)\tilde{v}_L(\tau,n)\|_{L^2_{\tau,n}} + \|\langle n\rangle^{2s-1+}L^{-1}\chi_{L\gtrsim |n|}\tilde{v}_L(\tau,n)\|_{\ell^2_nL^1_\tau}\\
\sim \left\|\left(\langle n\rangle^{2s-1+}L^{-1/2+}\chi_{L\ll|n|} + L^{-1/3}\chi_{L\gtrsim |n|}\right)\tilde{v}_L(\tau,n)\right\|_{\ell^2_{n,\tau}}\\
+ \|\langle n\rangle^{2s-1+}L^{-1}\chi_{L\gtrsim |n|}\tilde{v}_L(\tau,n)\|_{\ell^2_nL^1_\tau}.
\end{multline*}
We then let $\|\cdot\|_{X}$, be the functional we will estimate our nonlinearity in, as:
\[
\|v\|_{X} := \min(\|v\|_{\tilde{X}}, \|v\|_{X^{2s-1+, -1/2+}}).
\]
The following lemma allows us to handle the first term on the RHS of \eqref{eq:vn}, and also provides the intuition for the above functional.
\begin{lemma}\label{le:vkn}
Let $\mathcal{K}$ be such that $\|\mathcal{K}\|_{X} < \infty$, and define
\[
k_n (s)= \int_0^s \textnormal{Re}[\overline{v}_n\mathcal{K}_n](y)\,dy.
\]
Then, we must have
\[
    \left\|\cf^{-1}_n\left[\langle n\rangle^{2s-1+}v_n\mathcal{K}_n\right]\right\|_{Z^* }\lesssim \|v\|_{Z}^2\|\mathcal{K}\|_{X}.
\]
\end{lemma}
\begin{proof}
By the definition of $\|\cdot\|_X$, it suffices to show
\begin{multline*}
    \left\|\cf^{-1}_n\left[\langle n\rangle^{2s-1+}v_nk_n\right]\right\|_{Z^* }\\
    \lesssim \|v\|_{Z}^2 \bigg(\n{\left(\lan{n}^{2s-1+} L^{-\f{1}{2}+}\chi_{L\ll n } + L^{-\f{1}{3}-}\chi_{L\gtrsim n}\right)\wt{\mathcal{K}}(\tau,n)}{\ell^2_nL^2_\tau}\\
    + \n{\langle n \rangle^{2s-1+}L^{-1}\chi_{L\gtrsim n} \wt{\mathcal{K}_n}(\tau,n)}{L^2_n\ell^1_\tau}\bigg)
    \lesssim \|v\|_{Z}^2\|\mathcal{K}\|_{\tilde{X}},
\end{multline*}
and
\[
\left\|\cf^{-1}_n\left[\langle n\rangle^{2s-1+}v_nk_n\right]\right\|_{Z^* }\lesssim \|v\|_{Z}^2\|\mathcal{K}\|_{X^{2s-1+, -1/2+}}.
\]

We begin by using the definition of $Z$~norm, and observe that $\n{\cf^{-1}_n\left[\langle n\rangle^{2s-1+}v_nk_n\right]}{Z^* }$ can be written as:
\begin{align*}
 &\n{\pr{ |n|^{2s-1+}L^{-\f{1}{2}+} \chi_{L \ll n}+ |n|^{2s-1+}L^{-\f{1}{3}+}\chi_{L \gtrsim n} } \wt{v_n}*_\tau \wt{k_n}}{L^2_\tau\ell^2_n}\\
 & \lesssim  \n{ \pr{|n|^{2s-1+}L^{-\f{1}{2}+} \chi_{L \ll n}} \wt{v_n}*_\tau \wt{k_{n}}}{L^2_\tau\ell^2_n} +\n{\wt{v_n}*_\tau \wt{k_n}}{L^2_\tau\ell^2_n}        
\end{align*}
where we have assumed $s< \f{2}{3}$. The second term above is bounded by $\n{v}{L^{\infty}_t L^2_x}  \n{k}{L^2_{t,x}}$, so we focus on the first. Decomposing $k = k_1 + k_2$ as:

\[
    \wt{k}(n,\tau) = \chi_{\langle \tau\rangle \ll |n|}\wt{k}(n,\tau) + \chi_{\langle \tau\rangle\gtrsim |n|}\wt{k}(n,\tau) := \wt{k}_1(n,\tau) + \wt{k}_2(n,\tau),
\]

we observe that the term with $k_1$ can be bounded by $\n{v}{L^{\infty}_t L^2_x}  \n{k_1}{L^2_{t} H^{2s-1+}_x}$. For the term with $k_2$, we note that for $\tau =  \tau_1 + \tau_2$, 
\[
\tau - n^3 - \phi_n= (\tau_1 - n^3- \phi_n) + \tau_2,
\]
so that $\lan{\tau - n^3 - \phi_n} \ll n$ and $\lan{\tau_2} \gtrsim n$ would force $\lan{\tau_1 - n^3- \phi_n} \gtrsim n$.  Hence, this term with $k_2$ is bounded by
\begin{multline*}
\n{|n|^{2s-1+}  L^{-\f{1}{2}+} \wt{v_n} *_\tau \wt{k_{2,n}}} {L^2_{n,\tau}} \\
\lesssim \n{\f{\lan{n}^{2s-1+}}{\lan{\tau_1 - L_n - \phi_n}^{\f{2}{3}+}} (\lan{\tau_1 - L_n - \phi_n}^{\f{2}{3}+} \chi_{L_1 \gtrsim n} \wt{v_n}) *_\tau \wt{k_{2,n}}} {\ell^2_{n} L^\infty_\tau} 
\end{multline*}
which is bounded by  $\n{v}{Z} \n{k_2}{L^2_{t,x}}$ as long as $s< \f{5}{6}$. To summarize, we have obtained
\begin{equation}\label{eq:knest}
 \left\|\cf^{-1}_n\left[\langle n\rangle^{2s-1+}v_nk_n\right]\right\|_{Z^* }\lesssim \|v\|_{Z} \|k\|_{H^{2s-1+}_xL^2_t}.
\end{equation}

We now apply \cite[Lemma 2.1]{ginibre1997cauchy} to further reduce the $H^{2s-1}_xL^2_t$ term to
\begin{equation}\label{eq:kl2}
\|k\|_{L^2_{t,x}}\lesssim \|\textnormal{Re}(\overline{v}*_x\mathcal{K})\|_{H^{2s-1}_xH^{-1}_t} + \n{\langle \tau\rangle^{-1}\wt{\overline{v}}_n *_\tau \wt{\mathcal{K}_n}}{L^2_nL^1_\tau}.
\end{equation}
Denoting now the time-frequencies of $v_n$ and $\mathcal{NR}_n$ as $\tau_1$ and $\tau_2$ respectively, we observe
\[
\tau = -\tau_1+\tau_2 \implies \tau = -(\tau_1-n^3-\phi_n)+(\tau_2-n^3-\phi_n).
\]
Assume localization for $j=1,2$, such that $\lan{\tau_j - n^3 - \phi_n} \sim L_j$ as well as $\lan{\tau} \sim A$, and recall the definitions of $w_X$ and $w_Z$, \eqref{Equation: X weight} \&  \eqref{Equation: Z weight}. By H\"older's and Young's we conclude the proof of the $L^2$ term after observing that
\[
\frac{\langle n\rangle^{2s-1}}{w_Z(n,L_1)w_X(n,L_2)}\chi_{L_1\gtrsim L_2} + \frac{\langle n\rangle^{2s-1}}{w_X(n,L_2)A}\chi_{L_2\gg L_1} \lesssim 1.
\]
As for the $\ell^2_nL^1_\tau$ portion, we may conclude in a similar manner as above after observing that
\begin{align*}
\frac{\langle n\rangle^{2s-1}}{w_Z(n,L_1)w_X(n,L_2)}\chi_{L_1\sim L_2} + \frac{\langle n\rangle^{2s-1}}{w_X(n,L_2)A^{1/2-}}\chi_{|n|\gg L_1\sim A}&\lesssim 1.
\end{align*}
We note that the second portion of the norm in the statement of the lemma occurs due to $A\sim L_2\gtrsim |n|\gg L_1$, which results in that portion with no work.\\

As for the $X^{2s-1+, -1/2+}$ norm, we note that:
\[
\frac{L_2^{1/2-}}{w_Z(n,L_1)}\chi_{L_1\gtrsim L_2} + \frac{L_2^{1/2-}}{A^{1/2-}}\chi_{L_2\gg L_1} \lesssim 1,
\]
and hence we trivially obtain the final claim.
\end{proof}

    We actually gain a lot more in the $L^2$ portion. One should actually have the space
    \[
\tn{u_L} = \left\|\langle n\rangle^{2s-1+}\left(L^{-\frac12+}\chi_{L\ll |n|} + L^{-\frac23+}\chi_{L\gtrsim |n|}\right)\tilde{u}_L(\tau,n)\right\|_{L^2_{n,\tau}},
    \]
    however the case that $n_1\sim n_2\sim n_3\sim n$ yields the worst situation in which $|H_3|\sim \langle n\rangle$. In this situation we have $L\sim \langle n\rangle$, making these spaces essentially equivalent for the purpose of estimation.\\

Let $\mathcal{K}$ be such that $\|\mathcal{K}\|_{X} < \infty$, and hence in light of Lemma \ref{le:vkn} we observe
\begin{multline*}
    \left\|\eta(t)\int_0^t W_{t-s}\left(2i\f{\lan{n}^{2s}}{n} v_n \int_0^s  \textnormal{Re}[\overline{v_n}\mathcal{K}_n](y) \, dy + \mathcal{K}\right)\,ds\right\|_{Z}\\
    \lesssim \|v\|_{Z}^2\|\mathcal{K}\|_{X} + \|\mathcal{K}\|_{Z^*}\lesssim (\|v\|_{Z}^2+1)\|\mathcal{K}\|_{X}.
\end{multline*}

Linearity of the integral and the triangle inequality affords us the ability to apply the above display seperately to every term present within $\mathcal{NR}$ (see \eqref{eq:nr}); hence, it suffices to estimate everything in $\|\cdot\|_{X}$. While $\|\cdot\|_{X}$ lacks the triangle inequality, this property allows us to treat $X$ as if it did -- at least for the multilinear operators we are estimating in the Duhamel integral operator. We will, without confusion, write $\|\mathcal{NR}\|_{X}$ to denote the $X$ functional applied to every term in $\mathcal{NR}$.\\

Finally, we remark on the following term in the RHS of \eqref{eq:vn}:
\[ 2i\f{\lan{n}^{2s}}{n} v_n \left(  \text{Re}(\overline{f_n}w_n(0)) + \f{1}{2} \left|w_n(0) \right|^2\right). \]

Note that $w_n (0) = - [T^\ell(f,f)]_n  =: -\left[T^{\ell}(u,u)]_n \right|_{t=0}\in  H^{\f{3}{2} - s- }_x$ from Lemma~\ref{le:tell1}.  For $s< \f{2}{3}$, we have $\f{3}{2} - s > 2s -1$ so that the entire derivative weight, that is $\f{\lan{n}^{2s}}{n}$, can be absorbed into $w_n(0)$.  Hence, the following statement immediately follows without further work:
\[
\n{2i\eta(t)\mathcal{F}_{n}^{-1}\left(\f{\lan{n}^{2s}}{n} v_n \left(  \text{Re}(\overline{f_n}w_n(0)) + \f{1}{2} \left|w_n(0) \right|^2\right)\right)}{Z^{*}} \lesssim \n{v}{Z} (\|f\|_{L^2}\n{u}{Y}+\n{u}{Y}^4),
\]
for $\eta\in\mathcal{C}^\infty_0(\mathbb{R})$.
\subsection{Estimate of the nonlinearlity}

The nonlinearity $\mathcal{NR}$ is quite extensive and consists of many parts. More specifically, its parts are listed in \eqref{eq:nr}.  For this section we assume that we have the system of equations given by $u = h + v$ together with the integral equation derived from \eqref{eq:vn}. This will allow us to substitute terms for $u$, but means that our fixed point equation will need to have this substitution reflected. We ignore this minor detail in favor of increased readability.\\

We begin with the third term on this list. 

\begin{lemma}
Let $s < \f{2}{3}$ and $u\in Y$, then $\mathcal{N}^h$ of Equation \ref{definition: Nh} satisfies
\[
\n{\mathcal{N}^h(u,u)}{X^{2s-1+,-\frac12+}}\lesssim \n{u}{Y}^2. 
\]
\end{lemma}
\begin{proof}
We denote the external and internal frequencies by $(\tau,n)$, $(\tau_j, n_j)$ for $j=1,2$.  Also, we localize corresponding functions so that
\[
L \sim \lan{\tau - L_n}, \qquad L_j \sim \lan{\tau_j - L_{n_j}}
\]

We know that at least two elements of $\{L, L_1, L_2\}$ must be $\gg |n n_1 n_2|$. We remark that whenever $L_i\gtrsim |n|$ for $i= 1,2$ we find $J^{\frac13-}u\in X^{0,\frac13+}$. Thus, we must either have
\begin{align*}
    (L_1L_2)^{\frac13+}\langle n_{1}\rangle^{\frac13-}\langle n_{2}\rangle^{\frac13-}&\gtrsim \langle n\rangle^{\frac23+}\langle n_{1}\rangle \langle n_{2}\rangle\\
    &\gtrsim \langle n_{\max}\rangle^{\frac53-}\langle n_{\min}\rangle\qquad \mbox{or}\\
    L^{\frac12-}L_{\max}^{\frac13+}\langle n_{\min}\rangle^{\frac13-}&\gtrsim \langle n_{\max}\rangle^{\frac53-}\langle n_{\min}\rangle^{\frac76-}.
\end{align*}
Taking the worse of these two bounds, $\langle n_{\max}\rangle^{\frac53-}\langle n_{\min}\rangle$, we find
\[
\frac{\langle n\rangle^{2s-1}\sigma(\mathcal{N}^h)}{\langle n_{\max}\rangle^{\frac53-}\langle n_{\min}\rangle}\lesssim \frac{\langle n_{\max}\rangle^{2s}\langle n_{\min}\rangle^{s}}{\langle n_{\max}\rangle^{\frac53-}\langle n_{\min}\rangle}\lesssim \langle n_{\max}\rangle^{3s-\frac{13}{6}+}\langle n_{\min}\rangle^{-\frac12-}.
\]
It follows that
\begin{multline*}
\langle n\rangle^{2s-1}\sigma(\mathcal{N}^h)\\
\lesssim \langle n_{\min}\rangle^{-\frac12-}\left((L_1L_2)^{\frac13+}\langle n_{1}\rangle^{\frac13-}\langle n_{2}\rangle^{\frac13-} + L^{\frac12-}\max(L_{1}^{\frac13+}\langle n_{1}\rangle^{\frac13-},L_{2}^{\frac13+}\langle n_{2}\rangle^{\frac13-})\right),
\end{multline*}
and we may close by duality and the observation that for any permutation $\sigma$ of $\{1,2,3\}$:
\[
\left|\int_{\mathbb{R}\times\mathbb{T}}\big(\nab^{-\frac12-}f_{\sigma(1)}\big)f_{\sigma(2)}f_{\sigma(3)}\,dtdx\right|\lesssim \|f_{1}\|_{L^2_{x,t}}\|f_{2}\|_{L^2_{x,t}}\|f_{3}\|_{L^\infty_tL^2_x}.
\]
\end{proof}
Next, we estimate $\mathcal{M}$ in the $X$ functional.

\begin{lemma}\label{le:m}
Let $s < \f{2}{3}$, $u \in Y$, $v\in Z$, $h = T^{\ell}(u,u)$ as in Equation \ref{Definition: Tl definition}, and $u = h+v$. Then $\mathcal{M}$, defined by the nonresonant portion of 
\[
\mathcal{M}(u,u,u) = T^\ell(\mathcal{N}(u,u),u)
\]
as in Equation \ref{Definition: M definition}, satisfies
\[
\n{\mathcal{M}(u,u,u)}{X}\lesssim \|u\|_{Y}^3 + \|u\|_{Y}^3\|v\|_{Z} + \|u\|_{Y}^4
\]
\end{lemma}

Before proceeding with the proof of this lemma, which will follow by case work, we make some preliminary observations and reductions. Observe that the symbol, $\sigma(\mathcal{M})$, satisfies
\begin{equation}\label{eq:sym3}
\si(\mathcal{M}) = \f{ \lan{n_1}^s \lan{n_2}^s \lan{n_3}^s}{in_3\lan{n}^{s}  },
\end{equation}
and that it excludes the trilinear resonant interaction $H_3 := (n_1+ n_2) (n_2 + n_3) (n_3 + n_1) = 0$.  We now denote       
\begin{align*}
    L&= \langle \tau-n^3-\phi_n\rangle   \\
    L_i&= \langle \tau_i-n_i^3-\phi_{n_i}\rangle \quad \text{ for } i = 1,2,3  \\
    \wt{L_1}&= \langle \tau_1+\tau_2 -(n_1+n_2)^3-\phi_{n_1+n_2}\rangle,
\end{align*}
and remark that, in terms of modulation restriction, $\wt{L_1}$ and $L_3$ are limited to low modulation:
\begin{equation}\label{eq:modres}
\max\{L, \wt{L_1}, L_3\} \sim |n(n_1+n_2) n_3|.
\end{equation}

We now decompose $\mathbb{Z}^3$ into three separate regions, given below by Cases 1-3:\\

\textbf{Case 1}: $|n_{1}| \gg \max\{|n_{2}|, |n_3|\}$\\
\textbf{Case 2}: $|n_1| \sim |n_2| \sim |n_3|$\\
\textbf{Case 3}: $|n_1| \sim |n_2| \gg |n_3|$.\\

We note that these cases are not exhaustive.  However, all remaining frequency interactions are easier and hence are omitted for brevity.  For instance, we could also have $|n_1| \sim |n_3| \gg |n_2|$, but considering that $\si(\cm)$ has a full factor of $|n_3|$ in the denominator, we can see that this estimate will be much easier than its counterpart stated in Case 3.\\

We denote $L_{\max} = \max \{L_1, L_2, L_3, L\}$, so that we must have 
\[
L_{\max} \gtrsim \langle H_3\rangle  + |\phi_n-\phi_{n_1}-\phi_{n_2}-\phi_{n_3}|\gtrsim \langle H_3\rangle,
\]
as $\langle H_3\rangle\gtrsim \max(|n_1|,|n_2|,|n_3|)\gg \max(|n_1|,|n_2|,|n_3|)^{2s-1}$ for $s < 1$. We now perform an additional decomposition by defining the following subcases:\\

\textbf{Subcase A}: $L\sim L_{\max}$\\  
\textbf{Subcase B}: $L_j \sim L_{\max}$ for some $j=1,2,3$.\\

In Subcase A, $L \gtrsim H_3$, means that the external modulation $L$ and frequency $n$ satisfy $L\gtrsim |n|$.  Here, we utilize the $\tilde{X}$ portion of $X$ and only need to estimate the high modulation portion of $\tilde{X}$: 
\[
\|L^{-\frac13+}\chi_{L\gtrsim |n|}\wt{\mathcal{M}}(\tau,n)\|_{\ell^2_nL^2_\tau} + \|L^{-1}\langle n\rangle^{2s-1}\chi_{L\gtrsim |n|}\wt{\mathcal{M}}(\tau,n)\|_{\ell^2_nL^1_\tau}
\]
Hence, for Subcase A, the first thing we must do is show the symbol bound:
\begin{equation}\label{eq:sym1}
\f{\si(\mathcal{M})}{L^{\f{1}{3}-}} \lesssim \f{ \lan{n_1}^s \lan{n_2}^s }{\lan{n}^{s} \lan{n_3}^{1-s}\langle H_3\rangle^{\f{1}{3}-} }\lesssim \prod_{j=1}^3\langle n_j\rangle^{0-}.
\end{equation}

Once this is shown, Lemma~\ref{le:l6} and H\"older's inequality implies for $\tilde{\nu}(\tau,n) = \langle n\rangle^{0-}|\tilde{u}(\tau,n)|$: 
\[
\|\mathcal{M}(u,u,u)\|_{X^{0,-1/3+}}\lesssim \n{\nu^3}{L^2_{t,x}} \lesssim \n{\nu}{L^6_{t,x}}^3 \lesssim \n{u}{Y}^3,
\]
which is enough to obtain the desired estimate. The second portion we must show is:
\begin{equation}\label{eq:sym12}
\f{\langle n\rangle^{2s-1+}\si(\mathcal{M})}{L} \lesssim \f{ \lan{n_1}^s \lan{n_2}^s }{\lan{n}^{s} \lan{n_3}^{1-s}\langle H_3\rangle}\lesssim \frac{1}{\langle n_i+n_j \rangle^{1+}},
\end{equation}
for some $i,j\in \{1,2,3\}$. With this, we obtain:
\begin{multline*}
\|L^{-1}\langle n\rangle^{2s-1+}\chi_{L\gtrsim |n|}\wt{\mathcal{M}}(\tau,n)\|_{\ell^2_nL^1_\tau}\lesssim \left\|\tilde{u}*\wt{\left(\nab^{-1-} u^2\right)}\right\|_{\ell^2_nL^1_\tau}\\
\lesssim \|u\|_{Y}\sum_{m}\frac{1}{\langle m\rangle^{1+}}\sum_{m=m_1+m_2}\|\tilde{u}(\tau,m_1)\|_{L^1_\tau}*_n\|\tilde{u}(\tau,m_2)\|_{L^1_\tau},
\lesssim \|u\|_{Y}^3
\end{multline*}
which is the desired estimate.\\

In Subcase B,  we have either $L \ll |n|$ or $L\gtrsim |n|$.  For this case we utilize the second portion of the $X$ functional, $X^{2s-1+, -1/2+}$. Specifically, since we seek well-posedness for $s < \frac23$ we make the further reduction that we consider only the $X^{\frac13, -\frac12+}$ norm.\\

Since we assume $L_j \gtrsim \langle H_3\rangle \gtrsim |n_j|$ for some $j=1,2,3$, we gain $\lan{n_j}^{-\f{1}{3}+} L_j^{-\f{1}{3}-}$ for the symbol estimate.   Consider the symbol estimate:
\begin{equation}\label{eq:sym2}
\f{\lan{n}^{\f{1}{3}} \si(\mathcal{M})}{\lan{n_j}^{\f{1}{3}-} L_j^{\f{1}{3}-}} \lesssim \f{\lan{n_1}^s \lan{n_2}^s }{\lan{n}^{1-s}\lan{n_3}^{1-s} \lan{n_j}^{\f{1}{3}-} \langle H_3\rangle^{\f{1}{3}-} }\lesssim L^{0-}\langle n\rangle^{0-}\prod_{j=1}^3\langle n_j\rangle^{0-}.
\end{equation}
If this symbol bound is shown, then for 
\[
\tilde{\nu}(\tau,n) = \langle n\rangle^{0-}|\tilde{u}(\tau,n)|,\qquad \mbox{and}\qquad \tilde{\eta}(\tau,n) = \langle n\rangle^{\frac13-}L^{\frac13+}|\tilde{u}(\tau, n)|,
\]
we find the reduction:
\[
\|\mathcal{M}(u,u,u)\|_{X^{\frac13,-\frac12+}}\lesssim \left\| \nu^2 \eta \right\|_{X^{0-,-\frac12-}};
\]
which is easily bound using duality and an $L^6_{x,t}L^2_{x,t}L^6_{x,t}L^6_{x,t}$ estimate.\\

Note that the extra $L^{0-}$ in the required bound \eqref{eq:sym2} and the extra $\langle n_j\rangle^{0-}$ factors in both bounds can easily be obtained by borrowing from other modulation factors and noting that we will always have space in the spatial derivatives due to $s < \frac13$. Hence, with this understood, we simply show that the symbols are bounded by $1$.\\

\begin{remark}
    In Case 3, we will show that we must live in Subcase B. For a portion of this subcase we will use the relationship $u = h + v$, but we will eventually end up closing our argument using a fixed point as in Section \ref{Section: Proof Thm1}. Because of this, we technically have the spatial frequency dependent decomposition of $\mathcal{M}$ given by 
    \[
\mathcal{M} = \mathcal{M}^{(1,2)} + \mathcal{M}^{(3)},
    \]
    where the superscripts denote Cases 1,2,3. Our full equation will then have
    \[
\mathcal{M}^{(1,2)}(u,u,u) + \mathcal{M}^{(3)}(h+v, u, u)
    \]
    in place of $\mathcal{M}$. We omit this technicality for readability.
\end{remark}
\begin{proof}[Proof of Lemma \ref{le:m}]
With the above observations in place, we proceed with the proof of Lemma \ref{le:m}. This proof follows by the cases dictated above, and we estimate each case separately.\\
\textbf{Case 1:} $|n_{1}| \gg \max\{|n_2|, |n_3|\}$\\

Here, we must have $|n|\sim |n_1|$ and
\[
|H_3| \sim |(n_1+n_2)(n_2+n_3)(n_3+n_1)| \sim \lan{n_1}^2 \langle n_2+n_3\rangle
\]

\textbf{Case 1A:} $L\sim L_{\max}$\\

By \eqref{eq:sym1}, 
\[
\f{\si(\mathcal{M})}{L^{\f{1}{3}-}} \sim \f{ \lan{n_1}^s \lan{n_2}^s}{\lan{n}^{s}  \lan{n_3}^{1-s}\langle H_3\rangle ^{\f{1}{3}-} }\lesssim \f{ \lan{n_2}^s }{\lan{n_1}^{\f{2}{3}}\lan{n_3}^{1-s} \langle n_2+n_3\rangle^{\f{1}{3}-} }
\]
which is bounded if $s< 1$. Similarly, \eqref{eq:sym12} becomes:
\[
\frac{\langle n\rangle^{2s-1+}\sigma(\mathcal{M})}{L}\lesssim \frac{\langle n\rangle^{2s-1+}\langle n_1\rangle^{s}\langle n_2\rangle^{s}}{\langle n\rangle^{s}\langle n_3\rangle^{1-s}\langle H_3\rangle}\lesssim \frac{\langle n_2\rangle^{3s-3+}}{\langle n_3\rangle^{1-s}\langle n_2+n_3\rangle^{1+}},
\]
which will be bounded for $s < 1$.\\ 

\textbf{Case 1B:} $L_j \sim L_{\max}$ for some $j=1,2,3$\\

By \eqref{eq:sym2},
\[
\f{\lan{n}^{\f{1}{3}}\si(\mathcal{M})}{\lan{n_j}^{\f{1}{3}-} L_j^{\f{1}{3}+} } \sim \f{ \lan{n_1}^{s} \lan{n_2}^{s}  }{\lan{n}^{s - \f{1}{3}} \lan{n_3}^{1-s} \lan{n_j}^{\f{1}{3}-} \langle H_3\rangle ^{\f{1}{3}+} } \lesssim \f{ \lan{n_2}^{s} }{\lan{n_1}^{\f{1}{3}} \lan{n_3}^{1-s} \langle n_2+n_3\rangle^{\f{1}{3}+} }
\]
Which is bounded if $s < \f{2}{3}$, even whilst ignoring the extra gain in $\lan{n_j}^{-\f{1}{3}+}$. \\

\textbf{Case 2:} $|n_1| \sim |n_2| \sim |n_3|$\\

This is when $H_3$ can be small.  We can safely ignore the case when $|n|\ll |n_1|$, since this will imply $|H_3| \sim \lan{n_1}^3$ which gives a lot of gain from modulation.  Hence, we can assume $|n|\sim |n_1| \sim |n_2| \sim |n_3|$.\\

The worst case here is when only one factor of $H_3$ is large.  Without loss of generality, we assume that $\lan{n_1+n_2}\sim \lan{n}$ so that
\[
|H_3| \sim \lan{n} \langle n_2+n_3\rangle \langle n_3+n_1\rangle
\]

\textbf{Case 2A:} $L \sim L_{\max}$.\\

By \eqref{eq:sym1},
\[
\f{\si(\mathcal{M})}{L^{\f{1}{3}-} } \sim \f{ \lan{n}^{2s-\f{4}{3}+} }{ \langle n_1+n_2\rangle^{\f{1}{3}-} \langle n_2+n_3\rangle^{\f{1}{3}-} } 
\]
This is bounded for $s<\f{2}{3}$. Additionally, the \eqref{eq:sym12} bound becomes:
\[
\frac{\langle n\rangle^{2s-1+}\sigma(\mathcal{M})}{L}\lesssim \frac{\langle n\rangle^{s-1}\langle n_1\rangle^{s}\langle n_2\rangle^{s}\langle n_3\rangle^{s-1}}{\langle H_3\rangle }\lesssim \frac{\langle n\rangle^{4s-3+}}{\langle n_2+n_3\rangle^{1+}},
\]
which will be bounded for $s < \frac23$.\\

\textbf{Case 2B:} $L_j \sim L_{\max}$ for some $j=1,2,3$.\\

By \eqref{eq:sym2}, we have the chain of inequalities:
\[
\f{\lan{n}^{\f{1}{3}}\si(\mathcal{M})}{\lan{n_j}^{\f{1}{3}- }L_{j}^{\f{1}{3}-} } \sim \f{\lan{n}^{2s-1 +} }{ \langle H_3\rangle ^{\f{1}{3}-} }\lesssim \f{\lan{n}^{2s -\f{4}{3}+}}{ \langle n_1+n_2\rangle^{\f{1}{3}-} \langle n_2+n_3\rangle^{\f{1}{3}-} } 
\]
which is bounded for $s< \f{2}{3}$.\\

\textbf{Case 3:} $|n_1| \sim |n_2| \gg |n_3|$\\

In this main case, we have
\[
|H_3| = |(n_1+n_2)(n_2+n_3)(n_3+n_1)| \sim \lan{n_1}^2 \langle n_1+n_2\rangle
\]
We can assume that $|n_1+n_2| \ll |n_1|$ (which means $|n|\ll |n_1|$),  since otherwise we would get sufficient gain from $H_3$.  In Case 3, it is very important to rule out two possibilities: that is, $L\sim L_{\max}$ and $L_3\sim L_{\max}$.  In fact, these trilinear estimates will fail due to unbounded symbols, and this is the main reason for devising modulation restriction on the normal form $T^\ell$.  To see this, we denote the following:
    \begin{align*}
        L &= \langle \tau - n^3\rangle \\ 
        L_j &= \langle \tau_j - n_j^3\rangle \qquad \text{ for } j = 1,2,3\\  
        \wt{L_1} &= \langle (\tau_1 + \tau_2) - (n_1+n_2)^3 \rangle      
    \end{align*}

Recall that $\cm(u,u,u) = T^{\ell}(\mathcal{N}(u,u),u)$ (minus the resonance), and the modulation restriction from $T^{\ell}$ necessitates 
\[
\max\{ L, \wt{L_1}, L_3 \} \sim |n (n_1 + n_2) n_3|
\]
Hence, if either $L\sim L_{\max}$ or $L_3 \sim L_{\max}$, we must have 
\[
\lan{n_1}^2 \langle n_1+n_2\rangle \sim \langle H_3\rangle  \lesssim L_{\max} \lesssim  \langle n\rangle \langle n_1 + n_2\rangle \langle n_3\rangle
\]
But this implies $\langle n_1\rangle^2 \lesssim \langle n\rangle  \langle n_3\rangle $, which contradicts the frequency assumption for Case 3.  Hence, we must have $\max\{L_1, L_2\} = L_{\max} \gg \max\{L, L_3\}$.  By symmetry, we will assume that $L_1 = L_{\max} \sim \langle H_3\rangle $, which places us in \textbf{Case 3B}.\\

Here, we will need to use the decomposition $u = v+h$ for the first entry of  $\cm(\cdot, \cdot, \cdot)$.  Recall that $\chi_{L\gtrsim n} v \in X^{0,\f{2}{3}+}$ and $h= T^\ell(u,u)$.  For the term where the first entry is $v$, the symbol estimate is obtained by replacing the denominator on the LHS of \eqref{eq:sym2} by $L_1^{\f{2}{3}-}$, which yields:
\[
\f{\lan{n}^{\f{1}{3}} \si(\mathcal{M})}{ L_1^{\f{2}{3}-}} \lesssim \f{\lan{n_1}^{2s}  }{\lan{n}^{s-\f{1}{3}}\lan{n_3}^{1-s} \langle H_3\rangle^{\f{2}{3}-}  } \lesssim \f{\lan{n_1}^{2s-\f{4}{3}+} }{\lan{n}^{s-\f{1}{3}}\lan{n_3}^{1-s} \lan{n_1+n_2}^{\f{2}{3}-}  }
\]
which is bounded for $s < \f{2}{3}$. \\

Finally, we need to consider the case when the first entry is $h$.  Here, the trilinear estimate fails due to the unbounded symbol.  To overcome this issue, we write $h = T^\ell(u,u)$ and denote  $\mathcal{M}_1 :=\mathcal{M}(h,u,u) = T^{\ell} (\cn (T^{\ell}(u,u), u), u)$.  We will transition to a quadrilinear estimate for this case.  Note

\begin{align*}
\si(\cm_1) &= \f{ \lan{n_1}^s \lan{n_2}^s\lan{n_3}^s \lan{n_4}^s}{i n_1 n_2 n_4 \lan{n}^s}\\
&\times \chi(n_1+n_2+n_3, \tau_1+\tau_2+\tau_3; n_4) \chi(n_4, \tau_4; n_1+n_2+n_3) \chi (n_1,\tau_1; n_2)\chi(n_2,\tau_3; n_1)
\end{align*}
with the frequency restriction $\lan{n_1 + n_2} \sim |n_3| \gg |n_4|$.   The modulation restriction is included because we need this to establish cancellation via symmetry in one special case where $n_1 + n_3 =0$.  We divide this estimate into three sub-subcases:\\

\textbf{Subsubcase (i):} $|n_1|\sim |n_3| \gg \max\{|n_2|, |n_4|\}$ and $n_1 + n_3 = 0$;\\
\textbf{Subsubcase (ii):} $|n_1| \sim |n_3| \gg \max\{|n_2|, |n_4|\}$ and $n_1+n_3 \neq 0$;\\
\textbf{Subsubcase (iii):} $|n_1| \sim |n_2| \sim |n_3|$.\\

As we are no longer in the trilinear regime, we need to return to the $X^*$ norm.  As in the other Subcase B estimates, we can begin by estimating $\ds \n{\lan{n}^{\f{1}{3}} \lan{n}^{2s-1} \wt{\mathcal{M}_1}}{\ell^2_n L^2_\tau}$.  Note:
\[
\lan{n}^{\f{1}{3}} \si(\mathcal{M}_1) \lesssim \f{\lan{n_3}^s }{\lan{n}^{s-\f{1}{3}} \lan{n_1}^{1-s} \lan{n_2}^{1-s} \lan{n_4}^{1-s}}
\]
It would be insufficient to show that the RHS is bounded by 1, since we need to close a quadrilinear estimate in the absence of $L^8_{t,x}$ embedding.  But, it is sufficient to bound the RHS by $\lan{n_{\min}}^{-\f{1}{2}-}$, which enables us to pull out one factor of $u$ in $L^{\infty}_{t,x}$ and apply Sobolev embedding to place it in $L^\infty_t L^2_x$.  Then, we can use the same argument as the trilinear estimate in Subcase B to bound the remaining terms.  We denote
\begin{align*}
    L&= \langle \tau-n^3-\phi_n\rangle   \\
    L_i&= \langle \tau_i-n_i^3-\phi_{n_i}\rangle \quad \text{ for } i = 1,2,3,4   \\
    L_{\max} &= \max\{L, L_1, L_2, L_3\}\\
\end{align*}
We define $H_4$ to be the dispersion without $\phi_n's:$
\[
H_4 = (n_1+n_2+n_3+n_4)^3 - n_1^3 - n_2^3 - n_3^3 - n_4^3.
\]

\textbf{Case 3B(i):} $|n_1|\sim |n_3| \gg \max\{|n_2|, |n_4|\}$ and $n_1 + n_3 = 0$\\

Under this special case, $n = n_2 + n_4$ and $n_1^3 + n_3^3 =0$, so that $H_4 = (n_2+n_4)^3 - n_2^3 - n_4^3 = n_2 n_4(n_2+n_4)$. Even with the help of an $n_1$ factor in the denominator, the best we can estimate this symbol by is $\lan{n_3}^{2s -1}$, which is unbounded for $s>\f{1}{2}$.  However, there is a cancellation structure for this term.  We write $n_1 = N$ and $n_3 = -N$:
\[
\si(\mathcal{M}_1) = \f{ \lan{N}^{2s} \lan{n_2}^s  \lan{n_4}^s}{iN \lan{n}^s n_2 n_4} \chi(n_2, \tau_1 +\tau_2+\tau_3; n_4) \chi(n_4, \tau_4; n_2)   \chi(N, \tau_1; n_2) \chi(n_2,\tau_2; N)
\]
We want to make the symbol completely symmetric with respect to $(\tau_1, N)$ and $(\tau_3, -N)$.  To this end, we decompose $\cm_1$ as  $\cm_1 = \cm_{11} + \cm_{12}$, where:
\[
\si(\cm_{11}) = \si(\cm_1) \chi(-N, \tau_3; n_2); \qquad  \si(\cm_{12}) = \si(\cm_1) (1- \chi(-N, \tau_3; n_2)).
\]
First, consider $\cm_{11}$:
\begin{align*}
  \cf_{t,x}[\cm_{11}] &(\tau,n)  =  \int \sum \f{ \lan{N}^{2s} \lan{n_2}^s  \lan{n_4}^s}{iN \lan{n}^s n_2 n_4} \chi(n_2, \tau_1 +\tau_2+\tau_3; n_4) \chi(n_4, \tau_4; n_2) \\
                    &\times \chi(N, \tau_1; n_2) \chi(-N,\tau_3; n_2) \chi(n_2,\tau_2; N) \wt{u}(\tau_1, N) \wt{u}(\tau_2,n_2) \wt{u}(\tau_3, -N) \wt{u}(\tau_4, n_4),
\end{align*}
where, with the change variables $N\mapsto -N$ and $\tau_1 \mapsto \tau_3$, we may also write
\begin{align*}
  \cf_{t,x}[\cm_{11}] &(\tau,n)  =  -\int \sum \f{ \lan{N}^{2s} \lan{n_2}^s  \lan{n_4}^s}{N \lan{n}^s n_2 n_4} \chi(n_2, \tau_1 +\tau_2+\tau_3; n_4) \chi(n_4, \tau_4; n_2) \\
                    &\times \chi(-N, \tau_3; n_2) \chi(N,\tau_1; n_2) \chi(n_2,\tau_2; -N) \wt{u}(\tau_3, -N) \wt{u}(\tau_2,n_2) \wt{u}(\tau_1, N) \wt{u}(\tau_4,n_4).
\end{align*}

If we average over these two expressions we obtain:
\begin{multline*}
  \cf_{t,x}[\cm_{11}] (\tau,n)  =  \int \sum \f{ \lan{N}^{2s} \lan{n_2}^s  \lan{n_4}^s}{N \lan{n}^s n_2 n_4} \chi(n_2, \tau_1 +\tau_2+\tau_3; n_4) \chi(n_4, \tau_4; n_2)\\\times\chi(-N, \tau_1; n_2) \chi(N,\tau_3; n_2)
                      \left\{\f{ \chi(n_2,\tau_2; N) - \chi(n_2,\tau_2; -N)}{2}\right\} \wt{u}(\tau_1, -N) \\
                    \times\wt{u}(\tau_2,n_2) \wt{u}(\tau_3, N) \wt{u}(\tau_4,n_4),
\end{multline*}
where the support conditions on $\chi(n_2,\tau_2; N) - \chi(n_2,\tau_2; -N)$ forces at least one of the following conditions to fail:
\[
\lan{\tau_2 - n_2^3 - \phi_{n_2}} \lesssim |(n_2 + N) n_2 N| \quad \text{ or } \quad \lan{\tau_2 - n_2^3 - \phi_{n_2}} \lesssim |(n_2 - N) n_2 N| . 
\]
Since we have assumed $N \gg |n_2|$ in this case, we must have
\[
L_2 := \lan{\tau_2 - n_2^3 - \phi_{n_2}}  \gtrsim  N^2 |n _2|.
\]

Similarly, if we consider $\cm_{12}$, we see again by the support of $1- \chi(-N, \tau_3; n_2)$ that
\[
L_3 := \lan{\tau_3 - n_3^3 -\phi_{n_3} } \gg |(n_2+ n_3) n_2 n_3| \sim N^2 |n_2|.
\]
Thus, in all of the above scenarios we find
\[
\f{\lan{n}^{\f{1}{3}} \si(\mathcal{M}_1)}{\max\{L_2, L_3\}^{\f{1}{3}+}} \lesssim \f{\langle N\rangle^{2s-1} }{\lan{n}^{s-\f{1}{3}}  \lan{n_2}^{\f{4}{3}-s+} \lan{n_4}^{1-s}  \langle N\rangle^{\f{2}{3}+}  } \lesssim \lan{n_{\min}}^{-\f{1}{2}-},
\]
as long as $s \leq \f{5}{6}$.  \\

\textbf{Case 3B(ii):} $|n_1| \sim |n_3| \gg \max\{|n_2|, |n_4|\}$ and $n_1+n_3 \neq 0$ \\

In this case, we write
\[
H_4 = 3 (n_1+ n_2)(n_2+n_3)(n_3+n_1) + 3n_4 (n-n_4)n.
\]
If $\lan{n_1 + n_3}  \sim |n_1| \sim |n_3|$, then the first term of $H_4$ is $O(\lan{n_1}^3)$ whereas the second term contains a factor of $n_4$ and is much smaller.  Hence, $|H_4| \sim \lan{n_1}^3$.  If $\lan{n_1 + n_3}  \ll |n_1| \sim |n_3|$, then $|n| \ll |n_1|$, and the all three factors of $|n_4 (n - n_4) n|$ are very small in comparison to $n_1$, while the first term has at least two large factors comparable to $n_1$.  Hence, $|H_4| \sim \lan{n_1}^2 \langle n_1 + n_3\rangle \neq 0$.  In either of these cases, we can gain from $H_4$ at least $\f{2}{3}+$ derivatives in $n_1$, and the symbol can be easily bounded by $\lan{n_{\min}}^{-\f{1}{2}-}$.  We omit the details.\\

\textbf{Case 3B(iii):} $|n_1| \sim |n_2| \sim |n_3|$\\

In this case,
\[
\lan{n}^{\f{1}{3}} \si(\mathcal{M}_1) \lesssim \f{\lan{n_1}^{3s-2} }{\lan{n}^{s-\f{1}{3}}  \lan{n_4}^{1-s}}\lesssim \lan{n_{\min}}^{-\f{1}{2}-}
\]
even without using any modulation weight.  Hence, this estimate can be shown as before.
\end{proof}

We have two resonant terms left, that is, the first two terms of \eqref{eq:nr}.  We proceed with the first term on the list.
\begin{lemma}\label{le:r}
Let $s<\f{2}{3}$, $v\in Z$, and $h\in Y\cap L^2_tH^{1-}_x$. Then $\mathcal{R}$ of Equation \ref{definition: Mod Rest. Resonant} satisfies
\[
\n{\rr(v + h,v + h, v + h) - \rr (v,v,v)}{X^{2s-1+,-\frac12+}} \lesssim \|h\|_{L^2_tH^{-1}_x}\left(\|h\|_{Y} + \|v\|_{Y}\right)^2.
\]
\end{lemma}
\begin{proof}
Note that at least one $h$ is present in each term in the expansion.  By Lemma~\ref{le:tell1}, $h \in L^2_t H^{1-}_x$.  Note that $\si(\mathcal{R}) \sim \lan{n}^{2s-1}$, and 
\[
\frac{\langle n\rangle^{2s-1+}\sigma(\mathcal{R})}{\langle n\rangle^{1-}} \lesssim \f{\lan{n}^{4s-2+}}{\lan{n}^{1-}}\lesssim 1.
\]
Hence, the Sobolev weights are bounded when $s< \f{3}{4}$.  We can then close this by ignoring the modulation and utilizing the convolution structure in $x$:
\begin{equation}\label{Equation: Resonant Base Estimate}
\|f_1*_xf_2*_xf_3\|_{L^2_{x,t}}\lesssim \left\|\|f_1\|_{L^2_x}\|f_2\|_{L^2_x}\|f_3\|_{L^2_x}\right\|_{L^2_t}\lesssim \|f_i\|_{L^2_{x,t}}\|f_j\|_{Y}\|f_k\|_{Y},
\end{equation}
for any permutation $i,j,k$ of $1,2,3$. Thus we may always place $h$ into $L^2_{x,t}$, and this portion is concluded.
\end{proof}

Finally, we estimate the second term in \eqref{eq:nr}.

\begin{lemma}
Let $s < \frac23$ and $u\in Y$. Then $\mathcal{R}^h$ of Equation \ref{definition: Mod Rest. Resonant} satisfies
\[
\left\|\mathcal{R}^h(u,u,u)\right\|_{X^{2s-1+,-\frac12+}} \lesssim \|u\|_{Y}^3.
\]
\end{lemma}

\begin{proof}
Since we are filtering out $\mathcal{R}^\ell$, at least two of the modulation weights (out of $L$, $\wt{L_1}$, $L_3$) should be high, with respect to the bilinear dispersive symbol.  That is, say:
\[
\max\{L, L_3\}\gtrsim \wt{L_1} \gg |n(n_1+n_2)n_3|.
\]

Since $n_1 + n_2 \neq 0$ and the other two cases are symmetric, we consider the case $n_2 + n_3=0$, in which case $n_1 = n$. As we are estimating this in $X^{2s-1+, -\frac12+}$ for $s < \frac23$, we are free to instead estimate this in a more symmetric fashion in $X^{\frac13, \frac12+}$. From this, we see that
\[
\f{\langle n\rangle^{\frac13}\si(\mathcal{R}^h)}{\max(L^{\f{1}{3}-}, L_3^{\frac13-})} \lesssim   \f{ \langle n\rangle^{\frac13}\langle n_3\rangle^{2s-1}}{\langle n\rangle^{\frac13-}\langle n-n_3\rangle^{\frac13-}\langle n_3\rangle^{\frac13-}}\lesssim 1,
\]
for any $s < \frac23$. We may then close utilizing almost the same estimate as \eqref{Equation: Resonant Base Estimate}.
\end{proof}

This concludes the estimate of $\mathcal{NR}$.  We have shown:
\begin{equation}\label{eq:nrest}
\n{\mathcal{NR}}{X} \lesssim  C \pr{\n{v}{Z}, \n{u}{Y}}.
\end{equation}

\begin{remark}\label{Remark: Additional smoothing}
    We note that in our estimate of $\mathcal{NR}$, Lemma \ref{le:vkn}, we included an extra $\varepsilon$ factor of $\langle n\rangle$. By $\varepsilon$ pushing we have actually obtained the following: for any $s < \frac{2}{3}$ there is $\varepsilon = \varepsilon(s)$ such that $\|\nab^\varepsilon\mathcal{NR}\|_{X}\lesssim P(\|v\|_{Z}, \|u\|_{Y}, \|f\|_{L^2})$ where $P$ is a polynomial with every term having degree at least $2$.
\end{remark}

\section{Proof of Theorem~\ref{Theorem: 1}}\label{Section: Proof Thm1}

Let $\eta(t)\in\mathcal{S}$, $\eta(t)\equiv 1$ for $|t| < \frac12$, and $\eta(t)\equiv 0$ for $|t| > 1$. Then, \eqref{eq:kdv original} is equivalent to 
\begin{equation}\label{Equation: KdV with cutoff}
    \begin{cases}
        (\partial_t + \partial_x^3)u = \eta(t)^2\partial_x(u^2)\\
       u\vert_{t=0} =  u_0\in H^{-s},
    \end{cases}
\end{equation}
Note that this equation has solutions for small data $\|u_0\|_{H^{-s}}$ at high regularity by standard methods. We then perform the same modification to obtain the equivalent formulation \eqref{eq:kdv}, but with a time cutoff.\\

    While the addition of $\eta$ factors modify the above proofs, we note that they are in $H^\infty_t$. Additionally, all of our spaces are bounded with respect to this cutoff. These two combined facts allow us to disregard this minor technical difficulty.\\

    Additionally, there is a minor technical difficulty in the above with respect to the cancellation. We have technically assumed that for some $K\gg 1$ we have $\chi_{L\gtrsim K}(\eta T^\ell(\eta u, \eta u))$, but we may still perform the above after a reconstitution of the Littlewood-Paley blocks. This reduces us purely to the spatial constraint of Case 3, and hence the result follows again.\\

Now, if $u\big|_{|t|\leq 1}\equiv \tilde{u}\big|_{|t|\leq 1}$ then $h(u) = h(\tilde{u})$; hence, we are free to use restricted spaces with $T = 1$. Let
\[
B = B(f) := \{ \varphi = (\vp_1, \vp_2) \in Y_1\times Z_1\,|\,\tn{\varphi}:= \|\varphi_1\|_{Y_1} + \|\varphi_2\|_{Z_1} \leq C\|f\|_{L^2}\},
\]
where $Y_1$ and $Z_1$ are restricted $Y$ and $Z$ spaces for functions equivalent on $|t| \leq 1$ and $C$ depends solely on $\eta$. Define $\Gamma(\cdot,\cdot ) = \Gamma_f(\cdot, \cdot): Y_1\times Z_1\to Y_1\times Z_1$ via:
\[
\Gamma(\varphi_1, \vp_2) = (\Gamma_1(\varphi_1), \Gamma_2(\varphi_1)),
\]
where
\[
\left|\begin{array}{l}    \Gamma_1(\varphi) = h(\varphi_1) + \Gamma_2(\varphi_1)\\
    \Gamma_2(\varphi) = \eta(2t)W_t(f)f  - \eta(2t)W_T(f)h_0(\varphi_1)+ \eta(2t)\int_0^t \eta(2s)W_{t-s}(f)( \mbox{ RHS of }\eqref{eq:vn})\,ds.
\end{array}\right.
\]

We notice that if $u$ is ever a smooth solution to Equation~\ref{Equation: KdV with cutoff} for $|t| \leq 1$, then $(u,v)$ (as defined in the preceding sections) is a fixed point of $\Gamma$. Similarly, with $(u,v)$ a fixed point of the above, we have that $u$ solves \eqref{Equation: KdV with cutoff}.\\

Now, if $\varphi = (\vp_1, \vp_2)\in B$ then
\[
    \tn{\Gamma(\varphi_1, \vp_2)}\leq C_1\|f\|_{L^2} + P(\|f\|_{L^2}, \n{\vp_1}{Y_1}, \n{\vp_2}{Z_1})\leq C_3\|f\|_{L^2},
\]
as $P$ is an at-least degree $2$ polynomial and $\|f\|_{L^2}\ll 1$. Thus, we can guarantee $\Gamma: B\to B$ if $\n{f}{L^2}$ is sufficiently small. Also,
\[
    \tn{\Gamma(\varphi_1, \vp_2) - \Gamma(\psi_1, \psi_2)}\leq \|f\|_{L^2} P(\n{\vp_1 - \psi_1}{Y_1},\|\vp_2 - \psi_2\|_{Z_1}) \ll \tn{\vp - \psi}
\]
since $\|f\|_{L^2}\ll 1$. Banach's Fixed Point Theorem then yields existence and uniqueness of a solution in $B$, which implies that the solution $u$ satisfies $\lan{\nabla}^{-s} u\in Y_1 \subset C^0_t L^{2}_x$.\\

We now need to show continuous dependence on initial data. Before proceeding with the proof, we note the following Lemma.
\begin{lemma}\label{Lemma: AA basic lemmas}
Assume $\n{f}{L^2} \ll 1$ such that $\Gamma_f$ has a unique fixed point $\vp = (\vp_1, \vp_2)$ in $B(f)$. For $m \geq 1$, denote $P_{\leq m}\vp := (P_{\leq m}\vp_1, P_{\leq m} \vp_2)$ to be the spatial Fourier projections. Then the following statements hold:
    \begin{enumerate}
        \item For any $\varepsilon > 0$, there is an $m = m(f)\gg 1$ such that:
        \begin{equation}\label{Equation: Tail Smoothing Bound}
            \|(I - P_{\leq m})\vp_1\|_{C^0_tL^2_x} + \|(I - P_{\leq m})\vp_2\|_{C^0_tL^2_x} < \varepsilon.
        \end{equation}
        Furthermore, the dependence can be made uniform for a given family $\{f_{n}\}_{n\in\mathbb{N}}$ such that $\ds \lim_{n\to \infty} f_{n} = f$ in $L^2$.
        \item Given $m\geq 1$ and $0\leq s < \f{2}{3}$, there is $\nu = \nu(\|f\|_{L^2}, s) > 0$ so that for any $|t_1|, |t_2| < 1$:
        \begin{equation}\label{Equation: Equicontinuity bound}
            \|P_{\leq m}(\vp_1(t) - \vp_1(s))\|_{L^2_x} + \|P_{\leq m}(\vp_2(t) - \vp_2(s))\|_{L^2_x}
            < C(m, \|f\|_{L^2})|t-s|^\nu.
        \end{equation}
    \end{enumerate}
\end{lemma}
\begin{proof}
We first show that we have a smoothing property for high frequency cutoffs. Observe first that by Remark \ref{Remark: Additional smoothing}:
\begin{align*}
\|(I - P_{\leq m})\vp_2\|_{C^0_tL^2_x} &\leq C_1\|P_{> m}f\|_{L^2} + C_2\|P_{>m}\mathcal{NR}\|_{X}\\
&\leq C_1\|P_{>m}f\|_{L^2} + C_2m^{0-}\|J^{0+}\mathcal{NR}\|_{X} \\
&\leq C_1\|P_{>m}f\|_{L^2} + C_2m^{0-}P(\|f\|_{L^2}) < \varepsilon/2,
\end{align*}
as long as $m = m(f)\gg 1$. Now, turning our attention to the $\vp_1$ portion of the estimate, we see that by \ref{le:tell1} and the above:
\[
\|(I - P_{\leq m})\vp_1\|_{C^0_tL^2_x}\leq \|P_{> m}h(\vp_1)\|_{C^0_tL^2_x} + \|P_{> m}\vp_2\|_{C^0_tL^2_x} \leq C_1 m^{-\f{3}{2}+s+}\|\vp_1\|_{Y_1}^2 + \f{\varepsilon}{2}.
\]
Combining these, we obtain \eqref{Equation: Tail Smoothing Bound}. Additionally, the uniform dependence claim of $\{f_n\}_{n\in \mathbb{N}}$ follows trivially from the above.\\

We now seek to show \eqref{Equation: Equicontinuity bound}. We first handle the term $P_{\leq m}T(\vp_1,\vp_1)$. Rewriting this as a Fourier transform, we find:
\begin{multline*}
    \|P_{\leq m}T^\ell(\vp_1,\vp_1)(t) - P_{\leq m}T^\ell (\vp_1,\vp_1)(s)\|_{L^2_x}\\
    \leq \left\|\int_\mathbb{R} |e^{it\tau}-e^{is\tau}||\chi_{\leq m}(k)\mathcal{F}_{x,t}\left(T^\ell(\vp_1,\vp_1)\right)(\tau, k)|\,d\tau\right\|_{\ell^2_k}\\
    \leq |t-s|^{\nu}\left\|\int_\mathbb{R} |\tau|^\nu|\chi_{\leq m}(k)\mathcal{F}_{x,t}\left(T^\ell(\vp_1,\vp_1)\right)(\tau, k)|\,d\tau\right\|_{\ell^2_k}.
\end{multline*}

Considering the modulation restriction of $T^\ell(\cdot, \cdot)$, we note that 
\[
|\tau - k^3| \lesssim k n_1 (k - n_1) \implies |\tau|\lesssim \max(k^3, |kn_1 (k - n_1)|)
\]
where $n_1$ is the spatial frequency of the first entry $\vp_1$.    For $\nu< \f{1}{9}$ (for instance),
\[
|\tau|^\nu|\mathcal{F}_{x,t}\left(T^\ell(\vp_1,\vp_1)\right)(\tau, k)|\lesssim \langle k\rangle^{-1/3}\mathcal{F}_{x,t}\left(\nab^{-2/9}\vp_1\right)*_{\tau,k}\mathcal{F}_{x,t}\left(\nab^{-2/9}\vp_1\right)(\tau,k),
\]
and hence we may control this term by
\begin{multline*}
\left\|\int_\mathbb{R} |\tau|^\nu|\chi_{\leq m}(k)\mathcal{F}_{x,t}\left(T^\ell (\vp_1,\vp_1)\right)(\tau, k)|\,d\tau\right\|_{\ell^2_k}\\
\lesssim \left\|\int_\mathbb{R}\langle k\rangle^{-1/3}\mathcal{F}_{x,t}\left(\nab^{-2/9}\vp_1\right)*_{\tau,k}\mathcal{F}_{x,t}\left(\nab^{-2/9}\vp_1\right)(\tau,k)\,d\tau\right\|_{\ell^2_k}\lesssim \|\vp_1\|_{Y}^2.
\end{multline*}

We turn to $P_{\leq m}\vp_2$, where we again go to the Fourier transform to find:
\begin{align*}
    \|P_{\leq m}\vp_2(t) - P_{\leq m}\vp_2(s)\|_{L^2_x} &\leq \left\|\int_\mathbb{R} |e^{it\tau}-e^{is\tau}||\mathcal{F}_{x,t}\left({P_{\leq m}\vp_2}\right)(k,\tau)|\,d\tau\right\|_{\ell^2_k}\\
    &\leq |t-s|^\nu \left\|\int_\mathbb{R} |\tau|^\nu|\mathcal{F}_{x,t}\left(P_{\leq m}\vp_2\right)(k,\tau)|\,d\tau\right\|_{\ell^2_k}\\
    &\lesssim |t-s|^\nu m^{3\nu}\n{\lan{\tau - k^3 -\phi_k}^{\nu} \cf_{x,t} (P_{\leq m} \vp_2)}{\ell^2_k L^1_\tau}\\
    &\lesssim |t-s|^\nu m^{3\nu}\n{\lan{\tau - k^3 -\phi_k}^{\f{1}{2}+\nu +} \cf_{x,t} (P_{\leq m} \vp_2)}{\ell^2_k L^2_\tau}.
\end{align*}
Say $\nu < \f{1}{6}$.  Then
\begin{multline*}
\lan{\tau - k^3 -\phi_k}^{\f{1}{2}+\nu +}\leq  \lan{\tau - k^3 -\phi_k}^{\f{1}{2} +} \lan{k}^{\nu}\chi_{L\ll |k|}\\
+  \lan{\tau - k^3 -\phi_k}^{\f{2}{3} +}\chi_{L\gtrsim |k|} \leq \lan{k}^\nu w_Z(k,L).
\end{multline*}
Hence, we obtain 
\[
\|P_{\leq m}\vp_2(t) - P_{\leq m}\vp_2(s)\|_{L^2_x}  \lesssim |t-s|^\nu m^{4\nu}\n{ \vp_2}{Z}.
\]

We conclude the proof of the second claim by combining the above two bounds and noting that $\vp\in B(f)$.
\end{proof}

 Notice that if $\{f_{n}\}\in (L^2)^\mathbb{N}$ and $\{(u_n, v_n)\}\in (Y_1\times Z_1)^{\mathbb{N}}$ are solutions emanating from $f_{n}$ for each $n$, then continuous dependence would follow from the statement that $\ds \lim_{n\to \infty} f_{n} =  f$ in $L^2_x$ implies $\ds \lim_{n\to \infty} (u_n, v_n) = (u, v)$ in $(C^0_t([-1,1], L^2_x))^2$.\\

The standard definition of well-posedness requires the data-to-solution map be continuous from $B$ to $X_\delta$ with the $C^0_tH^s_x([-\delta,\delta]\times\mathbb{T})$ topology. As these spaces are first countable, sequential continuity is equivalent to continuity; hence, it suffices to show sequential continuity.\\

We now proceed with the proof of the continuity of the data-to-solution map. 

\begin{lemma}
    Let $f\in L^2_x$. If $\tilde{\Gamma}:L^2_x\to (C^0_tL^2_x)^2$ is the map which takes $f \mapsto (u,v)$, where $(u,v)$ is the unique fixed point of $\Gamma_f (\cdot, \cdot)$ in $B(f)$, then $\tilde{\Gamma}$ is a continuous map.
\end{lemma}
\begin{proof}
    We first note that it is sufficient to show this for $f_{n}\to f$ in $L^2$, where $f_n\in C^\infty$. Let $\{\varphi_n\}_{n\in \mathbb{N}} = \left\{(u_{n}, v_{n})\right\}_{n\in\mathbb{N}}$ be the fixed points of $\Gamma_{f_{n}}$ in $B(f_{n})$. For, $m\gg 1$, let $P_{\leq m}$ be as in Lemma \ref{Lemma: AA basic lemmas}, and $\varphi_{nm} = P_{\leq m}\varphi_n$. As $f_{n}\to f$, we have that $\{\varphi_{nm}\}_{n,m}$ is a uniformly bounded family in $C^0_t L^2_x$. Furthermore, for any $m > 1$ we have by \eqref{Equation: Equicontinuity bound} and the above that $\{\varphi_{nm}\}_{n}$ has a convergent subsequence in $(C^0_tL^2_x)^2$.  A diagonal argument and uniqueness concludes the argument.
\end{proof}

\end{document}